\documentclass[10pt]{article}

\usepackage{lscape,graphicx}
\usepackage{graphicx}

\usepackage{fullpage}
\usepackage{amssymb}
\usepackage{amsmath}

\usepackage{tikz}
\usetikzlibrary{arrows}
\tikzstyle{block}=[draw opacity=0.7,line width=1.4cm]

\usepackage{amsthm}
\newtheorem{theorem}{Theorem}[section]
\newtheorem{lemma}[theorem]{Lemma}

\usepackage{capt-of}

\begin{document}
\bibliographystyle{agsm}

\title{Zeta Functional Analysis\\
%{\it Journal of Integer Sequences}\\
%\large Version 1.34}
%\date{}
}
\author{Michael A. Idowu}
\maketitle

\begin{abstract}

We intimate deeper connections between the Riemann zeta and gamma functions than often reported and further derive a new formula for expressing the value of $\zeta(2n+1)$ in terms of zeta at other fractional points. This paper also establishes and presents new expository notes and perspectives on zeta function theory and functional analysis. In addition, a new fundamental result, in form of a new function called omega $\Omega(s)$, is introduced to analytic number theory for the first time. This new function together with some of its most fundamental properties and other related identities are here disclosed and presented as a new approach to the analysis of sums of generalised harmonic series, related alternating series and polygamma functions associated with Riemann zeta function. 

\end{abstract}

%\tableofcontents

\section{Introduction}
The Riemann zeta function is defined by the generalised harmonic series 
\begin{equation}
\zeta(s)= \sum^{\infty}_{k=1} \frac{1}{k^s}, 
\end{equation}
where $s = \sigma + it $ and $\sigma > 1$. In his 1859 paper \cite{Rie59}, Riemann introduced the functional equation
\begin{equation}  \frac{\zeta(s)\Gamma( \frac{s}{2})}{\pi^{\frac{s}{2}}}  =  \frac{\zeta(1-s)\Gamma( \frac{1-s}{2})}{\pi^{\frac{1-s}{2}}} 
\end{equation}
which suggests replacing the value of s by 1-s without changing the result of the outcome after substitution.
Our imagination is tickled by this property to investigate 
\begin{equation} 
 \log(\frac{\zeta(s)\Gamma( \frac{s}{2})}{\pi^{\frac{s}{2}}})  =  \log(\frac{\zeta(1-s)\Gamma( \frac{1-s}{2})}{\pi^{\frac{1-s}{2}}})  
\end{equation} 
hoping that we might find an elementary expression for the value of $\Gamma(1-s)$. Observe 
\begin{equation}
 \log(\frac{\zeta(s)}{\zeta(1-s)})  =  \log(\frac{\Gamma( \frac{1-s}{2})}{\Gamma( \frac{s}{2})}) +(s-\frac{1}{2})\log(\pi) = \log( \frac{(2\pi)^s}{\pi}\sin(\frac{\pi s}{2})\Gamma(1-s)    ),
\end{equation}
which implies
\begin{equation}\label{Formula1}
\log(\frac{\zeta(s)}{\zeta(1-s)})  =  \log(\frac{\Gamma( \frac{1-s}{2})}{\Gamma( \frac{s}{2})}) +(s-\frac{1}{2})\log(\pi) = s\log( 2\pi)-\log(\pi)+ \log(\sin(\frac{\pi s}{2}))+\log(\Gamma(1-s)).
\end{equation}
Without hesitation, we are induced to introduce $1-s$ in place of $s$ into the last equation to produce:
\begin{equation}\label{Formula2}
\log(\frac{\zeta(1-s)}{\zeta(s)})  =  \log(\frac{\Gamma( \frac{s}{2})}{\Gamma( \frac{1-s}{2})}) +(\frac{1}{2}-s)\log(\pi) = (1-s)\log( 2\pi)-\log(\pi)+ \log(\cos(\frac{\pi s}{2}))+\log(\Gamma(s)),
\end{equation}
similar to Riemann's idea. We digress from the original motivation of finding an expression for $\Gamma(1-s)$ to envision many great uses for these simple formulae created in \ref{Formula1} and \ref{Formula2}, taking on great feats such as the challenge of finding a new formula for $\zeta(2n+1)$ resulting in a new development and progress towards tackling this long standing open problem in number theory.

%**************************************************************************************************************

\section{Polygamma function and Riemann zeta at odd integers}
\begin{lemma}
Assuming $n \ge1$  is an integer number,
\textcolor{blue}{
\begin{equation}
\frac{d^{(2n+1)}}{ds^{(2n+1)}}(\log(\Gamma(s))) = 
\frac{d^{(2n+1)}}{ds^{(2n+1)}}(\log(\frac{\zeta(1-s)}{\zeta(s)})) - \frac{d^{(2n+1)}}{ds^{(2n+1)}}(\log(\cos(\frac{\pi s}{2})))  
\end{equation}
}
\end{lemma}
\begin{proof}
According to \ref{Formula2}
\begin{equation}
\frac{d^{(2n+1)}}{ds^{(2n+1)}}(\log(\Gamma(s))) +\frac{d^{(2n+1)}}{ds^{(2n+1)}}(\log(\cos(\frac{\pi s}{2})))  = \frac{d^{(2n+1)}}{ds^{(2n+1)}}(\log(\frac{\zeta(1-s)}{\zeta(s)})).
\end{equation}
\end{proof}

\begin{lemma}
Assuming $n \ge1$ is an integer number,
\textcolor{blue}{
\begin{equation}
\frac{d^{(2n+1)}}{ds^{(2n+1)}}(\log(\Gamma(\frac{s}{2}))) - \frac{d^{(2n+1)}}{ds^{(2n+1)}}(\log(\Gamma(\frac{1-s}{2}))) = 
\frac{d^{(2n+1)}}{ds^{(2n+1)}}(\log(\frac{\zeta(1-s)}{\zeta(s)})). 
\end{equation}
}
\end{lemma}
\begin{proof}
Again, this results as a direct consequence of \ref{Formula2}.
\end{proof}

\begin{lemma}
The following identities are valid:
\textcolor{blue}{
\begin{equation}
\frac{d^{(2n+1)}}{ds^{(2n+1)}}(\log(\Gamma(s))) = \frac{d^{(2n+1)}}{ds^{(2n+1)}}(\log(\Gamma(\frac{s}{2}))) - \frac{d^{(2n+1)}}{ds^{(2n+1)}}(\log(\Gamma(\frac{1-s}{2}))) - \frac{d^{(2n+1)}}{ds^{(2n+1)}}(\log(\cos(\frac{\pi s}{2}))) ; 
\end{equation}
}
and
\begin{equation}
\frac{d^{(2n+1)}}{ds^{(2n+1)}}(\log(\Gamma(1-s))) = -\frac{d^{(2n+1)}}{ds^{(2n+1)}}(\log(\Gamma(\frac{s}{2}))) + \frac{d^{(2n+1)}}{ds^{(2n+1)}}(\log(\Gamma(\frac{1-s}{2}))) - \frac{d^{(2n+1)}}{ds^{(2n+1)}}(\log(\sin(\frac{\pi s}{2}))). 
\end{equation}
\end{lemma}
\begin{proof}
Again, this results as a direct consequence of \ref{Formula2}.
\end{proof}

\begin{theorem}
Assuming n is an integer number, the following is a formula expressing the value of $\zeta(2n+1)$.
\textcolor{blue}{
\begin{equation}
(-1)^{2n+1}(2^{2n+1}-1)\zeta(2n+1)\Gamma(2n+1)= \psi^{(2n)}(\frac{1}{2}) = 
\frac{d^{(2n+1)}}{ds^{(2n+1)}}(\log(\frac{\zeta(1-s)}{\zeta(s)}))\biggr \vert_{s \rightarrow \frac{1}{2}}  - \frac{d^{(2n+1)}}{ds^{(2n+1)}}(\log(\cos(\frac{\pi s}{2}))) \biggr \vert_{s \rightarrow \frac{1}{2}} 
\end{equation}
}
\end{theorem}
\begin{proof}
\begin{equation}
\frac{d^{(2n+1)}}{ds^{(2n+1)}}(\log(\Gamma(s))) \biggr \vert_{s \rightarrow \frac{1}{2}} = 
\frac{d^{(2n+1)}}{ds^{(2n+1)}}(\log(\frac{\zeta(1-s)}{\zeta(s)}))\biggr \vert_{s \rightarrow \frac{1}{2}}  - \frac{d^{(2n+1)}}{ds^{(2n+1)}}(\log(\cos(\frac{\pi s}{2}))) \biggr \vert_{s \rightarrow \frac{1}{2}} 
\end{equation}
\end{proof}

According to K.S. K\"{o}lbig \cite{Kol96}, special values of  polygamma function \cite{Nis11, Ido12b} defined as 
$\psi^{(s-1)}(x) = {d^{s-1}\over{dx^{s-1}}}\psi(x) = {d^{s}\over{dx^{s}}}\ln \Gamma(x)$ may be combined to compute (or compose) the value of $\zeta(2n+1)$, e.g. $\psi^{(k)}(\frac{1}{4})+\psi^{(k)}(\frac{3}{4})$ as in the case of:
\begin{equation}\label{main}
\zeta(n)= (-1)^n.{{( \psi^{(n-1)}({1\over4})+\psi^{(n-1)}({3\over4}))}\over{2^n.(2^n-1)}}.{1 \over{\Gamma(n)}}
\end{equation}.

\begin{theorem}
Let n be an integer number, then
\textcolor{blue}{
\begin{equation}
-\zeta(2n+1)= 
{{( \psi^{(2n)}({1\over4})+\psi^{(2n)}({3\over4}))}\over{2^{2n+1}.(2^{2n+1}-1)}}.{1 \over{\Gamma(2n+1)}}
= 
\frac{2
\frac{d^{(2n+1)}}{ds^{(2n+1)}}(\log(\frac{\zeta(1-s)}{\zeta(s)}))\biggr \vert_{s \rightarrow \frac{1}{4}}  + \frac{d^{(2n+1)}}{ds^{(2n+1)}}(\log(\tan(\frac{\pi s}{2}))) \biggr \vert_{s \rightarrow \frac{1}{4}} 
}
{
2^{2n+1}(2^{2n+1}-1)\Gamma(2n+1)
}
\end{equation}
}
\end{theorem}
\begin{proof}
Notice that 
\begin{equation}
\psi^{(2n)}({1\over4}) =\frac{d^{(2n+1)}}{ds^{(2n+1)}}(\log(\Gamma(s))) \biggr \vert_{s \rightarrow \frac{1}{4}} = 
\frac{d^{(2n+1)}}{ds^{(2n+1)}}(\log(\frac{\zeta(1-s)}{\zeta(s)}))\biggr \vert_{s \rightarrow \frac{1}{4}}  - \frac{d^{(2n+1)}}{ds^{(2n+1)}}(\log(\cos(\frac{\pi s}{2}))) \biggr \vert_{s \rightarrow \frac{1}{4}} 
\end{equation}
\begin{equation}
\begin{split}
\psi^{(2n)}({3\over4}) =\frac{d^{(2n+1)}}{ds^{(2n+1)}}(\log(\Gamma(s))) \biggr \vert_{s \rightarrow \frac{3}{4}} = 
\frac{d^{(2n+1)}}{ds^{(2n+1)}}(\log(\frac{\zeta(1-s)}{\zeta(s)}))\biggr \vert_{s \rightarrow \frac{3}{4}}  - \frac{d^{(2n+1)}}{ds^{(2n+1)}}(\log(\cos(\frac{\pi s}{2}))) \biggr \vert_{s \rightarrow \frac{3}{4}} \\
=
\frac{d^{(2n+1)}}{ds^{(2n+1)}}(\log(\frac{\zeta(1-s)}{\zeta(s)}))\biggr \vert_{s \rightarrow \frac{1}{4}}  +\frac{d^{(2n+1)}}{ds^{(2n+1)}}(\log(\sin(\frac{\pi s}{2}))) \biggr \vert_{s \rightarrow \frac{1}{4}} 
\end{split}
\end{equation}
Therefore
\begin{equation}
\psi^{(2n)}({1\over4})+\psi^{(2n)}({3\over4})
= 
2
\frac{d^{(2n+1)}}{ds^{(2n+1)}}(\log(\frac{\zeta(1-s)}{\zeta(s)}))\biggr \vert_{s \rightarrow \frac{1}{4}}  + \frac{d^{(2n+1)}}{ds^{(2n+1)}}(\log(\tan(\frac{\pi s}{2}))) \biggr \vert_{s \rightarrow \frac{1}{4}}.
\end{equation}
\end{proof}~\\
Clearly the following theorems are valid and do not require explicit proofs.

\begin{theorem}\label{maintheorem}
In general, 
\textcolor{blue}{
\begin{equation}
\begin{split}
\psi^{(2n)}(s)+\psi^{(2n)}(1-s)
= 
2
\frac{d^{(2n+1)}}{ds^{(2n+1)}}(\log(\frac{\zeta(1-s)}{\zeta(s)})) + \frac{d^{(2n+1)}}{ds^{(2n+1)}}(\log(\tan(\frac{\pi s}{2})))
\\
= 
-2
\frac{d^{(2n+1)}}{ds^{(2n+1)}}(\log(\frac{\zeta(s)}{\zeta(1-s)})) - \frac{d^{(2n+1)}}{ds^{(2n+1)}}(\log(\cot(\frac{\pi s}{2}))),
\end{split}
\end{equation}
}
\end{theorem}
where $s \notin \{0,1\}$.

\begin{theorem}\label{maintheorem2}
In general, 
\textcolor{blue}{
\begin{equation}
\begin{split}
\psi^{(2n)}(s)-\psi^{(2n)}(1-s)
= 
-\frac{d^{(2n+1)}}{ds^{(2n+1)}}(\log(\cos(\frac{\pi s}{2})\sin(\frac{\pi s}{2})))
\end{split}
\end{equation}
}
\end{theorem}
where $s \notin \{0,1\}$.

%********************************************************************
\section{Further analysis of zeta functional equations}

\begin{theorem}
Given that s is any (real of complex) number, except 0 and 1, then 
\textcolor{blue}{
\begin{equation}
\begin{split}
\log(\frac{\Gamma(\frac{1-s}{2})}{\Gamma(\frac{s}{2})}) =  \log(\frac{\zeta(\frac{s+1}{2})}{\zeta(1 - (\frac{s+1}{2}))} \frac{\zeta(\frac{s}{2})}{\zeta(1 - \frac{s}{2})} ) - (s-\frac{1}{2})\log( {\bf 2\pi})- \log(\frac{\sin(\frac{\pi (\frac{s+1}{2})}{2})}{\sin(\frac{\pi (1-\frac{s}{2})}{2})});\\
\log(\frac{\zeta(s)}{\zeta(1-s)}) =  \log(\frac{\zeta(\frac{s+1}{2})}{\zeta(1 - (\frac{s+1}{2}))} \frac{\zeta(\frac{s}{2})}{\zeta(1 - \frac{s}{2})} ) - (s-\frac{1}{2})\log( {\bf 2})- \log(\frac{\sin(\frac{\pi (\frac{s+1}{2})}{2})}{\sin(\frac{\pi (1-\frac{s}{2})}{2})})
\end{split}
\end{equation}
}
\end{theorem}
\begin{proof}
Observe $\log(\frac{\zeta(s)}{\zeta(1-s)}) -  s\log( 2\pi)+\log(\pi)- \log(\sin(\frac{\pi s}{2})) = \log(\Gamma(1-s))$ which implies $\log(\frac{\zeta(1-s)}{\zeta(s)}) -  (1-s)\log( 2\pi)+\log(\pi)- \log(\cos(\frac{\pi s}{2})) = \log(\Gamma(s))$. Since $ \frac{1-s}{2} = 1 - (\frac{s+1}{2})$, $ \frac{s}{2} = 1 - (1-\frac{s}{2})$, the following identities are valid:
\begin{equation}\label{log1}
\log(\Gamma(\frac{1-s}{2})) = \log(\Gamma(1 - (\frac{s+1}{2}))) = \log(\frac{\zeta(\frac{s+1}{2})}{\zeta(1 - (\frac{s+1}{2}))}) - (\frac{s+1}{2})\log( 2\pi)+\log(\pi)- \log(\sin(\frac{\pi (\frac{s+1}{2})}{2}));
\end{equation}
\begin{equation}\label{log2}
\log(\Gamma(\frac{s}{2})) = \log(\Gamma(1 - (1-\frac{s}{2}))) = \log(\frac{\zeta(1-\frac{s}{2})}{\zeta(1 - (1-\frac{s}{2}))}) - (1-\frac{s}{2})\log( 2\pi)+\log(\pi)- \log(\sin(\frac{\pi (1-\frac{s}{2})}{2})).
\end{equation}\\
Subtracting eq. \ref{log2} from \ref{log1} gives:

\[
\begin{split}
\log(\frac{\Gamma(\frac{1-s}{2})}{\Gamma(\frac{s}{2})}) =  \log(\frac{\zeta(\frac{s+1}{2})}{\zeta(1 - (\frac{s+1}{2}))} \frac{\zeta(\frac{s}{2})}{\zeta(1 - \frac{s}{2})} ) - (s-\frac{1}{2})\log( {\bf 2\pi})- \log(\frac{\sin(\frac{\pi (\frac{s+1}{2})}{2})}{\sin(\frac{\pi (1-\frac{s}{2})}{2})});\\
\log(\frac{\zeta(s)}{\zeta(1-s)}) =  \log(\frac{\zeta(\frac{s+1}{2})}{\zeta(1 - (\frac{s+1}{2}))} \frac{\zeta(\frac{s}{2})}{\zeta(1 - \frac{s}{2})} ) - (s-\frac{1}{2})\log( {\bf 2})- \log(\frac{\sin(\frac{\pi (\frac{s+1}{2})}{2})}{\sin(\frac{\pi (1-\frac{s}{2})}{2})})
\end{split}
\]
since $ \log(\frac{\Gamma(\frac{1-s}{2})}{\Gamma(\frac{s}{2})}) = \log(\frac{\zeta(s)}{\zeta(1-s)})  - (s-\frac{1}{2})\log \pi $.
\end{proof}~\\
The following identities demonstrate a combinatorial perspective on $\log(\frac{\zeta(s)}{\zeta(1-s)})$ decomposition:

%****************************************************************************************
\begin{lemma}\label{LemmaMulti}
\begin{equation}
\begin{split}
\log(\frac{\zeta(s+1)}{\zeta(-s)}) =  \log(\frac{\zeta(\frac{s+1}{2})}{\zeta(1 - (\frac{s+1}{2}))} \frac{\zeta(\frac{s+2}{2})}{\zeta(1 - \frac{s+2}{2})} ) - (s+\frac{1}{2})\log( 2)- \log(\frac{\sin(\frac{\pi (\frac{s+2}{2})}{2})}{\sin(\frac{\pi (1-\frac{s+1}{2})}{2})}); \\
\log(\frac{\zeta(s)}{\zeta(1-s)}) =  \log(\frac{\zeta(\frac{s}{2})}{\zeta(1 -\frac{s}{2})} \frac{\zeta(\frac{s+1}{2})}{\zeta(1 - (\frac{s+1}{2}))} ) - (s-\frac{1}{2})\log(2)- \log(\frac{\sin(\frac{\pi (\frac{s+1}{2})}{2})}{\sin(\frac{\pi (1-\frac{s}{2})}{2})}); \\
\log(\frac{\zeta(s-1)}{\zeta(2-s)}) =  \log(\frac{\zeta(\frac{s-1}{2})}{\zeta(1 - (\frac{s-1}{2}))} \frac{\zeta(\frac{s}{2})}{\zeta(1 - \frac{s}{2})} ) - (s-\frac{3}{2})\log( 2)- \log(\frac{\sin(\frac{\pi (\frac{s}{2})}{2})}{\sin(\frac{\pi (1-\frac{s-1}{2})}{2})}); \\
\log(\frac{\zeta(s - 2)}{\zeta(3-s)}) =  \log(\frac{\zeta(\frac{s-2}{2})}{\zeta(1 - (\frac{s-2}{2}))} \frac{\zeta(\frac{s-1}{2})}{\zeta(1 - \frac{s-1}{2})} ) - (s-\frac{5}{2})\log( 2)- \log(\frac{\sin(\frac{\pi (\frac{s-1}{2})}{2})}{\sin(\frac{\pi (1-\frac{s-2}{2})}{2})}).
\end{split}
\end{equation}
\end{lemma}~\\
Many important identities are derivable for this last set of equations.

\begin{theorem}
\textcolor{blue}{
\begin{equation}
\begin{split}
\log( \frac{\zeta(s+\frac{1}{2})}{\zeta(\frac{1}{2} - s)} \frac{\zeta(s)}{\zeta(1 -s)})  -\log(\frac{\sin(\frac{\pi}{2} (s+\frac{1}{2}))}{\sin(\frac{\pi}{2} (1-s))})  = \log( \frac{\zeta(s-\frac{1}{2})}{\zeta(\frac{3}{2} - s)} \frac{\zeta(s-1)}{\zeta(2 -s)}) -\log( \frac{(s-1)(s-\frac{1}{2})}{(2 \pi i)^2}
 - \log(\frac{\sin(\frac{\pi}{2} (s-\frac{1}{2}))}{\sin(\frac{\pi}{2} (2-s))}) 
\end{split}
\end{equation}
}
\end{theorem}
\begin{proof}
Note 
\[
\log(\frac{\zeta(s-2)}{\zeta(3-s)}) =  \log(\frac{\zeta(\frac{s-2}{2})}{\zeta(1 -\frac{s-2}{2})} \frac{\zeta(\frac{s-1}{2})}{\zeta(1 - (\frac{s-1}{2}))} ) - (s-\frac{5}{2})\log(2)- \log(\frac{\sin(\frac{\pi (\frac{s-1}{2})}{2})}{\sin(\frac{\pi (1-\frac{s-2}{2})}{2})}), 
\]
and \footnote{$\frac{\zeta(s-2)}{\zeta(3-s)} =  \frac{(s-2)(s-1)}{(2 \pi i)^2}\frac{\zeta(s)}{\zeta(1-s)} $ - an independent discovery by the author similar in comparison to one of Henrik Stenlund's results\cite{Hen11}. }.\\
\[ 
\log( \frac{(s-2)(s-1)}{(2 \pi i)^2}\frac{\zeta(s)}{\zeta(1-s)})  =  \log(\frac{\zeta(\frac{s-2}{2})}{\zeta(1 -\frac{s-2}{2})} \frac{\zeta(\frac{s-1}{2})}{\zeta(1 - (\frac{s-1}{2}))} ) - (s-\frac{5}{2})\log(2)- \log(\frac{\sin(\frac{\pi (\frac{s-1}{2})}{2})}{\sin(\frac{\pi (1-\frac{s-2}{2})}{2})}) 
\]

$\Rightarrow$
\[
\log(\frac{\zeta(s)}{\zeta(1-s)}) =  \log(\frac{\zeta(\frac{s}{2})}{\zeta(1 -\frac{s}{2})} \frac{\zeta(\frac{s+1}{2})}{\zeta(1 - (\frac{s+1}{2}))} ) - (s-\frac{1}{2})\log(2)- \log(\frac{\sin(\frac{\pi (\frac{s+1}{2})}{2})}{\sin(\frac{\pi (1-\frac{s}{2})}{2})})
\]

\begin{equation}
\begin{split}
\log(\frac{\zeta(\frac{s}{2})}{\zeta(1 -\frac{s}{2})} \frac{\zeta(\frac{s+1}{2})}{\zeta(1 - (\frac{s+1}{2}))} )  =  \log(\frac{\zeta(\frac{s-2}{2})}{\zeta(1 -\frac{s-2}{2})} \frac{\zeta(\frac{s-1}{2})}{\zeta(1 - (\frac{s-1}{2}))} ) +\\
+2\log(2)
-\log( \frac{(s-2)(s-1)}{(2 \pi i)^2} 
+\log(\frac{\sin(\frac{\pi (\frac{s+1}{2})}{2})}{\sin(\frac{\pi (1-\frac{s}{2})}{2})}) 
- \log(\frac{\sin(\frac{\pi (\frac{s-1}{2})}{2})}{\sin(\frac{\pi (1-\frac{s-2}{2})}{2})}) 
\end{split}
\end{equation}~\\
This suggests

\begin{equation}
\begin{split}
\log( \frac{\zeta(s+\frac{1}{2})}{\zeta(\frac{1}{2} - s)} \frac{\zeta(s)}{\zeta(1 -s)})  =  \log( \frac{\zeta(s-\frac{1}{2})}{\zeta(\frac{3}{2} - s)} \frac{\zeta(s-1)}{\zeta(2 -s)})  \\
-\log( \frac{(s-1)(s-\frac{1}{2})}{(2 \pi i)^2} 
+\log(\frac{\sin(\frac{\pi}{2} (s+\frac{1}{2}))}{\sin(\frac{\pi}{2} (1-s))}) 
- \log(\frac{\sin(\frac{\pi}{2} (s-\frac{1}{2}))}{\sin(\frac{\pi}{2} (2-s))}) 
\end{split}
\end{equation}
from which the identity

\[
\begin{split}
\log( \frac{\zeta(s+\frac{1}{2})}{\zeta(\frac{1}{2} - s)} \frac{\zeta(s)}{\zeta(1 -s)})  -\log(\frac{\sin(\frac{\pi}{2} (s+\frac{1}{2}))}{\sin(\frac{\pi}{2} (1-s))})  = \log( \frac{\zeta(s-\frac{1}{2})}{\zeta(\frac{3}{2} - s)} \frac{\zeta(s-1)}{\zeta(2 -s)}) -\log( \frac{(s-1)(s-\frac{1}{2})}{(2 \pi i)^2}
 - \log(\frac{\sin(\frac{\pi}{2} (s-\frac{1}{2}))}{\sin(\frac{\pi}{2} (2-s))}) 
\end{split}
\]
is derived.
\end{proof}

%***************************************************************************************
\subsection{Zeta at 1+s, -s, 1-s and s}

\begin{theorem}\label{logzeta}
\textcolor{blue}{
\begin{equation}
\begin{split}
 \log(\frac{\zeta^2(1+s)}{\zeta^2(-s)} ) = 
-\log(\frac{\zeta^2(1-s)}{\zeta^2(s)} ) 
-\log(\frac{s^2}{(2\pi i)^2} )+\\
- \log(\frac{\sin(\frac{\pi (\frac{s+2}{2})}{2})}{\sin(\frac{\pi (1-\frac{s+1}{2})}{2})})
+  \log(\frac{\sin(\frac{\pi (\frac{s+1}{2})}{2})}{\sin(\frac{\pi (1-\frac{s}{2})}{2})})
-\log(\frac{\sin(\frac{\pi (\frac{s}{2})}{2})}{\sin(\frac{\pi (1-\frac{s-1}{2})}{2})})
+\log(\frac{\sin(\frac{\pi (\frac{s-1}{2})}{2})}{\sin(\frac{\pi (1-\frac{s-2}{2})}{2})})
\end{split}
\end{equation}
}
\end{theorem}

\begin{proof}
According to lemma \ref{LemmaMulti},

\begin{equation}
\begin{split}
  \log(\frac{\zeta(s+1)}{\zeta(-s)}) -\log(\frac{\zeta(s)}{\zeta(1-s)})+\log(\frac{\zeta(s-1)}{\zeta(2-s)}) -\log(\frac{\zeta(s - 2)}{\zeta(3-s)})-\biggr( \log( \frac{\zeta(\frac{s+2}{2})}{\zeta(1 - \frac{s+2}{2})} )-\log( \frac{\zeta(\frac{s-2}{2})}{\zeta(1 - \frac{s-2}{2})} )\biggr)\\
=
-2\log 2 
- \log(\frac{\sin(\frac{\pi (\frac{s+2}{2})}{2})}{\sin(\frac{\pi (1-\frac{s+1}{2})}{2})})
+  \log(\frac{\sin(\frac{\pi (\frac{s+1}{2})}{2})}{\sin(\frac{\pi (1-\frac{s}{2})}{2})})
-\log(\frac{\sin(\frac{\pi (\frac{s}{2})}{2})}{\sin(\frac{\pi (1-\frac{s-1}{2})}{2})})
+\log(\frac{\sin(\frac{\pi (\frac{s-1}{2})}{2})}{\sin(\frac{\pi (1-\frac{s-2}{2})}{2})})
\end{split}
\end{equation}

The recursion relations between the first six terms may be combined to produce:
\begin{equation}
\begin{split}
\log(\frac{{\bf (s-1)s}}{(2\pi i)^2}\frac{\zeta^2(s+1)}{\zeta^2(-s)}) 
-\log(\frac{{\bf (s-2)(s-1)}}{(2\pi i)^2} \frac{\zeta^2(s)}{\zeta^2(1-s)})
-\biggr(\log( \frac{\zeta(\frac{s+2}{2})}{\zeta(1 - \frac{s+2}{2})} ) - \log( \frac{{\bf (\frac{s}{2}-1)\frac{s}{2}}}{(2\pi i)^2} \frac{\zeta(\frac{s+2}{2})}{\zeta(1 - \frac{s+2}{2})} )\biggr)\\
=
-2\log 2 
- \log(\frac{\sin(\frac{\pi (\frac{s+2}{2})}{2})}{\sin(\frac{\pi (1-\frac{s+1}{2})}{2})})
+  \log(\frac{\sin(\frac{\pi (\frac{s+1}{2})}{2})}{\sin(\frac{\pi (1-\frac{s}{2})}{2})})
-\log(\frac{\sin(\frac{\pi (\frac{s}{2})}{2})}{\sin(\frac{\pi (1-\frac{s-1}{2})}{2})})
+\log(\frac{\sin(\frac{\pi (\frac{s-1}{2})}{2})}{\sin(\frac{\pi (1-\frac{s-2}{2})}{2})})
\end{split}
\end{equation}

or alternatively,

\begin{equation}
\begin{split}
\log(\frac{{\bf (s-1)s}}{(2\pi i)^2}\frac{\zeta^2(s+1)}{\zeta^2(-s)}) 
-\log(\frac{{\bf (s-2)(s-1)}}{(2\pi i)^2} \frac{\zeta^2(s)}{\zeta^2(1-s)})
-\biggr( \log( \frac{(2\pi i)^2}{{\bf -\frac{s}{2}(1-\frac{s}{2})}} \frac{\zeta(\frac{s-2}{2})}{\zeta(1 - \frac{s-2}{2})} )-\log( \frac{\zeta(\frac{s-2}{2})}{\zeta(1 - \frac{s-2}{2})} )\biggr)\\
=
-2\log 2 
- \log(\frac{\sin(\frac{\pi (\frac{s+2}{2})}{2})}{\sin(\frac{\pi (1-\frac{s+1}{2})}{2})})
+  \log(\frac{\sin(\frac{\pi (\frac{s+1}{2})}{2})}{\sin(\frac{\pi (1-\frac{s}{2})}{2})})
-\log(\frac{\sin(\frac{\pi (\frac{s}{2})}{2})}{\sin(\frac{\pi (1-\frac{s-1}{2})}{2})})
+\log(\frac{\sin(\frac{\pi (\frac{s-1}{2})}{2})}{\sin(\frac{\pi (1-\frac{s-2}{2})}{2})})
\end{split}
\end{equation}
As a result of this aggregation

\begin{equation}\label{interim}
\begin{split}
\log(\frac{{\bf (s-1)s}}{(2\pi i)^2}\frac{\zeta^2(s+1)}{\zeta^2(-s)}) 
-\log(\frac{{\bf (s-2)(s-1)}}{(2\pi i)^2} \frac{\zeta^2(s)}{\zeta^2(1-s)})
+\log( \frac{{\bf (\frac{s}{2}-1)\frac{s}{2}}}{(2\pi i)^2} )\\
=
-2\log 2 
- \log(\frac{\sin(\frac{\pi (\frac{s+2}{2})}{2})}{\sin(\frac{\pi (1-\frac{s+1}{2})}{2})})
+  \log(\frac{\sin(\frac{\pi (\frac{s+1}{2})}{2})}{\sin(\frac{\pi (1-\frac{s}{2})}{2})})
-\log(\frac{\sin(\frac{\pi (\frac{s}{2})}{2})}{\sin(\frac{\pi (1-\frac{s-1}{2})}{2})})
+\log(\frac{\sin(\frac{\pi (\frac{s-1}{2})}{2})}{\sin(\frac{\pi (1-\frac{s-2}{2})}{2})})
\end{split}
\end{equation}
is produced. Further aggregation results in

\begin{equation}
\begin{split}
\log(\frac{{\bf s}}{\bf s-2}\frac{\zeta^2(1+s)}{\zeta^2(-s)} \frac{\zeta^2(1-s)}{\zeta^2(s)}) = 
-\log( \frac{{\bf (\frac{s}{2}-1)\frac{s}{2}}}{(2\pi i)^2} )
-2\log 2 \\
- \log(\frac{\sin(\frac{\pi (\frac{s+2}{2})}{2})}{\sin(\frac{\pi (1-\frac{s+1}{2})}{2})})
+  \log(\frac{\sin(\frac{\pi (\frac{s+1}{2})}{2})}{\sin(\frac{\pi (1-\frac{s}{2})}{2})})
-\log(\frac{\sin(\frac{\pi (\frac{s}{2})}{2})}{\sin(\frac{\pi (1-\frac{s-1}{2})}{2})})
+\log(\frac{\sin(\frac{\pi (\frac{s-1}{2})}{2})}{\sin(\frac{\pi (1-\frac{s-2}{2})}{2})})
\end{split}
\end{equation}

$\Rightarrow$

\begin{equation}
\begin{split}
\log(\frac{\zeta^2(1+s)}{\zeta^2(-s)} \frac{\zeta^2(1-s)}{\zeta^2(s)}) = 
-\log(\frac{{\bf s}}{\bf s-2})
-\log( \frac{{\bf (\frac{s}{2}-1)\frac{s}{2}}}{(2\pi i)^2} )
-2\log 2 \\
- \log(\frac{\sin(\frac{\pi (\frac{s+2}{2})}{2})}{\sin(\frac{\pi (1-\frac{s+1}{2})}{2})})
+  \log(\frac{\sin(\frac{\pi (\frac{s+1}{2})}{2})}{\sin(\frac{\pi (1-\frac{s}{2})}{2})})
-\log(\frac{\sin(\frac{\pi (\frac{s}{2})}{2})}{\sin(\frac{\pi (1-\frac{s-1}{2})}{2})})
+\log(\frac{\sin(\frac{\pi (\frac{s-1}{2})}{2})}{\sin(\frac{\pi (1-\frac{s-2}{2})}{2})})
\end{split}
\end{equation}

$\Rightarrow$

\begin{equation}
\begin{split}
 \log(\frac{\zeta^2(1+s)}{\zeta^2(-s)} \frac{\zeta^2(1-s)}{\zeta^2(s)}) = 
-\log( \frac{2{\bf s^2}}{\bf (s-2)} \frac{{\bf (\frac{s}{2}-1)}}{(2\pi i)^2} )+\\
- \log(\frac{\sin(\frac{\pi (\frac{s+2}{2})}{2})}{\sin(\frac{\pi (1-\frac{s+1}{2})}{2})})
+  \log(\frac{\sin(\frac{\pi (\frac{s+1}{2})}{2})}{\sin(\frac{\pi (1-\frac{s}{2})}{2})})
-\log(\frac{\sin(\frac{\pi (\frac{s}{2})}{2})}{\sin(\frac{\pi (1-\frac{s-1}{2})}{2})})
+\log(\frac{\sin(\frac{\pi (\frac{s-1}{2})}{2})}{\sin(\frac{\pi (1-\frac{s-2}{2})}{2})})
\end{split}
\end{equation}
and finally

\[
\begin{split}
 \log(\frac{\zeta^2(1+s)}{\zeta^2(-s)} ) = 
-\log(\frac{\zeta^2(1-s)}{\zeta^2(s)} ) 
-\log(\frac{s^2}{(2\pi i)^2} )+\\
- \log(\frac{\sin(\frac{\pi (\frac{s+2}{2})}{2})}{\sin(\frac{\pi (1-\frac{s+1}{2})}{2})})
+  \log(\frac{\sin(\frac{\pi (\frac{s+1}{2})}{2})}{\sin(\frac{\pi (1-\frac{s}{2})}{2})})
-\log(\frac{\sin(\frac{\pi (\frac{s}{2})}{2})}{\sin(\frac{\pi (1-\frac{s-1}{2})}{2})})
+\log(\frac{\sin(\frac{\pi (\frac{s-1}{2})}{2})}{\sin(\frac{\pi (1-\frac{s-2}{2})}{2})}).
\end{split}
\]
\end{proof}

%*****************************************************************************
\subsection{On the decomposition of the derivatives of logarithm of zeta}
Here we would like to prove the following

\begin{theorem}
\textcolor{blue}{
\begin{equation}
\begin{split}
\log(\frac{{\bf (s-2)(s-1)}}{(2\pi i)^2} \frac{\zeta^2(s)}{\zeta^2(1-s)}) \\
=
\log(  \frac{\bf \zeta(\frac{s+2}{2})}{\bf \zeta(1 - \frac{s+2}{2})}
 \frac{\zeta(\frac{s+1}{2})}{\zeta(1 - \frac{s+1}{2})} 
 \frac{\zeta(\frac{s}{2})}{\zeta(1 - \frac{s}{2})}
 \frac{\zeta(\frac{s-1}{2})}{\zeta(1 - \frac{s-1}{2})} )
+\log( \frac{{\bf (\frac{s}{2}-1)\frac{s}{2}}}{(2\pi i)^2} )
-({\bf 2s-3})\log 2 +\\
-  \log(\frac{\sin(\frac{\pi (\frac{s+1}{2})}{2})}{\sin(\frac{\pi (1-\frac{s}{2})}{2})})
-\log(\frac{\sin(\frac{\pi (\frac{s-1}{2})}{2})}{\sin(\frac{\pi (1-\frac{s-2}{2})}{2})})
\\
=
\log(  
 \frac{\zeta(\frac{s+1}{2})}{\zeta(1 - \frac{s+1}{2})} 
 \frac{\zeta(\frac{s}{2})}{\zeta(1 - \frac{s}{2})}
 \frac{\zeta(\frac{s-1}{2})}{\zeta(1 - \frac{s-1}{2})}
  \frac{\bf \zeta(\frac{s-2}{2})}{\bf \zeta(1 - \frac{s-2}{2})} )
-({\bf 2s-3})\log 2 
-  \log(\frac{\sin(\frac{\pi (\frac{s+1}{2})}{2})}{\sin(\frac{\pi (1-\frac{s}{2})}{2})})
-\log(\frac{\sin(\frac{\pi (\frac{s-1}{2})}{2})}{\sin(\frac{\pi (1-\frac{s-2}{2})}{2})})
\end{split}
\end{equation}
}
\end{theorem}

\begin{proof}
Again, according to lemma \ref{LemmaMulti},
\begin{equation}
\begin{split}
  \log(\frac{\zeta(s+1)}{\zeta(-s)}) +\log(\frac{\zeta(s)}{\zeta(1-s)})+\log(\frac{\zeta(s-1)}{\zeta(2-s)}) +\log(\frac{\zeta(s - 2)}{\zeta(3-s)})-
\biggr( 
\log( \frac{\zeta(\frac{s+2}{2})}{\zeta(1 - \frac{s+2}{2})} ) +
\log( \frac{\zeta(\frac{s-2}{2})}{\zeta(1 - \frac{s-2}{2})} )
\biggr)\\
-\biggr( 
\log( \frac{\zeta^2(\frac{s+1}{2})}{\zeta^2(1 - \frac{s+1}{2})} ) +
\log( \frac{\zeta^2(\frac{s}{2})}{\zeta^2(1 - \frac{s}{2})} )+
\log( \frac{\zeta^2(\frac{s-1}{2})}{\zeta^2(1 - \frac{s-1}{2})} )
\biggr)\\
=
-(4s-4)\log 2 
- \log(\frac{\sin(\frac{\pi (\frac{s+2}{2})}{2})}{\sin(\frac{\pi (1-\frac{s+1}{2})}{2})})
-  \log(\frac{\sin(\frac{\pi (\frac{s+1}{2})}{2})}{\sin(\frac{\pi (1-\frac{s}{2})}{2})})
-\log(\frac{\sin(\frac{\pi (\frac{s}{2})}{2})}{\sin(\frac{\pi (1-\frac{s-1}{2})}{2})})
-\log(\frac{\sin(\frac{\pi (\frac{s-1}{2})}{2})}{\sin(\frac{\pi (1-\frac{s-2}{2})}{2})})
\end{split}
\end{equation}

$\Rightarrow$

\begin{equation}
\begin{split}
\log(\frac{{\bf (s-1)s}}{(2\pi i)^2}\frac{\zeta^2(s+1)}{\zeta^2(-s)}) 
+\log(\frac{{\bf (s-2)(s-1)}}{(2\pi i)^2} \frac{\zeta^2(s)}{\zeta^2(1-s)})
-
\biggr( 
\log( \frac{\zeta(\frac{s+2}{2})}{\zeta(1 - \frac{s+2}{2})} ) +
\log( \frac{\zeta(\frac{s-2}{2})}{\zeta(1 - \frac{s-2}{2})} )
\biggr)\\
-\biggr( 
\log( \frac{\zeta^2(\frac{s+1}{2})}{\zeta^2(1 - \frac{s+1}{2})} ) +
\log( \frac{\zeta^2(\frac{s}{2})}{\zeta^2(1 - \frac{s}{2})} )+
\log( \frac{\zeta^2(\frac{s-1}{2})}{\zeta^2(1 - \frac{s-1}{2})} )
\biggr)\\
=
-(4s-4)\log 2 
- \log(\frac{\sin(\frac{\pi (\frac{s+2}{2})}{2})}{\sin(\frac{\pi (1-\frac{s+1}{2})}{2})})
-  \log(\frac{\sin(\frac{\pi (\frac{s+1}{2})}{2})}{\sin(\frac{\pi (1-\frac{s}{2})}{2})})
-\log(\frac{\sin(\frac{\pi (\frac{s}{2})}{2})}{\sin(\frac{\pi (1-\frac{s-1}{2})}{2})})
-\log(\frac{\sin(\frac{\pi (\frac{s-1}{2})}{2})}{\sin(\frac{\pi (1-\frac{s-2}{2})}{2})}).
\end{split}
\end{equation}
%*********************************
From this last equation we subtract eq. \ref{interim}:
\[
\begin{split}
\log(\frac{{\bf (s-1)s}}{(2\pi i)^2}\frac{\zeta^2(s+1)}{\zeta^2(-s)}) 
-\log(\frac{{\bf (s-2)(s-1)}}{(2\pi i)^2} \frac{\zeta^2(s)}{\zeta^2(1-s)})
+\log( \frac{{\bf (\frac{s}{2}-1)\frac{s}{2}}}{(2\pi i)^2} )\\
=
-2\log 2 
- \log(\frac{\sin(\frac{\pi (\frac{s+2}{2})}{2})}{\sin(\frac{\pi (1-\frac{s+1}{2})}{2})})
+  \log(\frac{\sin(\frac{\pi (\frac{s+1}{2})}{2})}{\sin(\frac{\pi (1-\frac{s}{2})}{2})})
-\log(\frac{\sin(\frac{\pi (\frac{s}{2})}{2})}{\sin(\frac{\pi (1-\frac{s-1}{2})}{2})})
+\log(\frac{\sin(\frac{\pi (\frac{s-1}{2})}{2})}{\sin(\frac{\pi (1-\frac{s-2}{2})}{2})})
\end{split}
\]
to yield

\begin{equation}
\begin{split}
2\log(\frac{{\bf (s-2)(s-1)}}{(2\pi i)^2} \frac{\zeta^2(s)}{\zeta^2(1-s)})
=\\
\biggr( 
\log( \frac{\zeta(\frac{s+2}{2})}{\zeta(1 - \frac{s+2}{2})} ) +
\log( \frac{\zeta(\frac{s-2}{2})}{\zeta(1 - \frac{s-2}{2})} )
\biggr)
+\biggr( 
\log( \frac{\zeta^2(\frac{s+1}{2})}{\zeta^2(1 - \frac{s+1}{2})} ) +
\log( \frac{\zeta^2(\frac{s}{2})}{\zeta^2(1 - \frac{s}{2})} )+
\log( \frac{\zeta^2(\frac{s-1}{2})}{\zeta^2(1 - \frac{s-1}{2})} )
\biggr)\\
+\log( \frac{{\bf (\frac{s}{2}-1)\frac{s}{2}}}{(2\pi i)^2} )
-({\bf 4s-6})\log 2 
-  2\log(\frac{\sin(\frac{\pi (\frac{s+1}{2})}{2})}{\sin(\frac{\pi (1-\frac{s}{2})}{2})})
-2\log(\frac{\sin(\frac{\pi (\frac{s-1}{2})}{2})}{\sin(\frac{\pi (1-\frac{s-2}{2})}{2})})
\end{split}
\end{equation}

which shows the decomposition relation of $\log(\frac{\zeta^2(s)}{\zeta^2(1-s)} ) $ as 

%\newpage

\begin{equation}
\begin{split}
2\log(\frac{{\bf (s-2)(s-1)}}{(2\pi i)^2} \frac{\zeta^2(s)}{\zeta^2(1-s)})
=\\
\biggr( 
\log(  \frac{{\bf (\frac{s}{2}-1)\frac{s}{2}}}{(2\pi i)^2} \frac{\zeta^2(\frac{s+2}{2})}{\zeta^2(1 - \frac{s+2}{2})} ) \biggr)
+\biggr( 
\log( \frac{\zeta^2(\frac{s+1}{2})}{\zeta^2(1 - \frac{s+1}{2})} ) +
\log( \frac{\zeta^2(\frac{s}{2})}{\zeta^2(1 - \frac{s}{2})} )+
\log( \frac{\zeta^2(\frac{s-1}{2})}{\zeta^2(1 - \frac{s-1}{2})} )
\biggr)\\
+\log( \frac{{\bf (\frac{s}{2}-1)\frac{s}{2}}}{(2\pi i)^2} )
-({\bf 4s-6})\log 2 
-  2\log(\frac{\sin(\frac{\pi (\frac{s+1}{2})}{2})}{\sin(\frac{\pi (1-\frac{s}{2})}{2})})
-2\log(\frac{\sin(\frac{\pi (\frac{s-1}{2})}{2})}{\sin(\frac{\pi (1-\frac{s-2}{2})}{2})})
\end{split}
\end{equation}

simplified to 

\begin{equation}
\begin{split}
\log(\frac{{\bf (s-2)(s-1)}}{(2\pi i)^2} \frac{\zeta^2(s)}{\zeta^2(1-s)})
=
\log(  \frac{\bf \zeta(\frac{s+2}{2})}{\bf \zeta(1 - \frac{s+2}{2})}
 \frac{\zeta(\frac{s+1}{2})}{\zeta(1 - \frac{s+1}{2})} 
 \frac{\zeta(\frac{s}{2})}{\zeta(1 - \frac{s}{2})}
 \frac{\zeta(\frac{s-1}{2})}{\zeta(1 - \frac{s-1}{2})} )\\
+\log( \frac{{\bf (\frac{s}{2}-1)\frac{s}{2}}}{(2\pi i)^2} )
-({\bf 2s-3})\log 2 
-  \log(\frac{\sin(\frac{\pi (\frac{s+1}{2})}{2})}{\sin(\frac{\pi (1-\frac{s}{2})}{2})})
-\log(\frac{\sin(\frac{\pi (\frac{s-1}{2})}{2})}{\sin(\frac{\pi (1-\frac{s-2}{2})}{2})})
\end{split}
\end{equation}

or

\begin{equation}
\begin{split}
\log(\frac{{\bf (s-2)(s-1)}}{(2\pi i)^2} \frac{\zeta^2(s)}{\zeta^2(1-s)})
=
\log(  
 \frac{\zeta(\frac{s+1}{2})}{\zeta(1 - \frac{s+1}{2})} 
 \frac{\zeta(\frac{s}{2})}{\zeta(1 - \frac{s}{2})}
 \frac{\zeta(\frac{s-1}{2})}{\zeta(1 - \frac{s-1}{2})}
  \frac{\bf \zeta(\frac{s-2}{2})}{\bf \zeta(1 - \frac{s-2}{2})} )\\
-({\bf 2s-3})\log 2 
-  \log(\frac{\sin(\frac{\pi (\frac{s+1}{2})}{2})}{\sin(\frac{\pi (1-\frac{s}{2})}{2})})
-\log(\frac{\sin(\frac{\pi (\frac{s-1}{2})}{2})}{\sin(\frac{\pi (1-\frac{s-2}{2})}{2})})
\end{split}
\end{equation}

depending on the identity

\begin{equation}
\begin{split}
  \frac{\bf \zeta(\frac{s-2}{2})}{\bf \zeta(1 - \frac{s-2}{2})} )=
\log( \frac{{\bf (\frac{s}{2}-1)\frac{s}{2}}}{(2\pi i)^2} ) +
\log(  \frac{\bf \zeta(\frac{s+2}{2})}{\bf \zeta(1 - \frac{s+2}{2})}.
\end{split}
\end{equation}

\end{proof}

%**************************************************************************************
\begin{lemma} The following identity is true.
\textcolor{blue}{
\begin{equation}\label{sinecosine}
\begin{split}
\log(\frac{\sin(\frac{\pi}{2} ({\bf \frac{s}{2}}+\frac{1}{2}))}{\cos(\frac{\pi}{2} ({\bf \frac{s}{2}}))})
=\log(\frac{\sin(\frac{\pi}{2} ({\bf \frac{s-4}{2}}+\frac{1}{2}))}{\cos(\frac{\pi}{2} ({\bf \frac{s-4}{2}}))}).
\end{split}
\end{equation}
}
\end{lemma}

\begin{proof}

\begin{equation}
\begin{split}
\log(\frac{\zeta^2(s)}{\zeta^2(1-s)})
=
\log(  
 \frac{\zeta(\frac{s+1}{2})}{\zeta(1 - \frac{s+1}{2})} 
 \frac{\zeta(\frac{s}{2})}{\zeta(1 - \frac{s}{2})}
 \frac{\zeta(\frac{s-1}{2})}{\zeta(1 - \frac{s-1}{2})}
  \frac{\bf \zeta(\frac{s-2}{2})}{\bf \zeta(1 - \frac{s-2}{2})} )\\
-\log(\frac{{\bf (s-2)(s-1)}}{(2\pi i)^2} )
-({\bf 2s-3})\log 2 
-  \log(\frac{\sin(\frac{\pi (\frac{s+1}{2})}{2})}{\sin(\frac{\pi (1-\frac{s}{2})}{2})})
-\log(\frac{\sin(\frac{\pi (\frac{s-1}{2})}{2})}{\sin(\frac{\pi (1-\frac{s-2}{2})}{2})})
\end{split}
\end{equation}
Recall the identity

\[
\log(\frac{\zeta(s)}{\zeta(1-s)}) =  \log(\frac{\zeta(\frac{s}{2})}{\zeta(1 -\frac{s}{2})} \frac{\zeta(\frac{s+1}{2})}{\zeta(1 - (\frac{s+1}{2}))} ) - (s-\frac{1}{2})\log(2)- \log(\frac{\sin(\frac{\pi (\frac{s+1}{2})}{2})}{\sin(\frac{\pi (1-\frac{s}{2})}{2})}),
\]
therefore

\begin{equation}
\begin{split}
\log(\frac{\zeta^2(\frac{s}{2})}{\zeta^2(1 -\frac{s}{2})} \frac{\zeta^2(\frac{s+1}{2})}{\zeta^2(1 - (\frac{s+1}{2}))} ) - (2s-1)\log(2)- 2\log(\frac{\sin(\frac{\pi (\frac{s+1}{2})}{2})}{\sin(\frac{\pi (1-\frac{s}{2})}{2})}) =\\
\log(  
 \frac{\zeta(\frac{s+1}{2})}{\zeta(1 - \frac{s+1}{2})} 
 \frac{\zeta(\frac{s}{2})}{\zeta(1 - \frac{s}{2})}
 \frac{\zeta(\frac{s-1}{2})}{\zeta(1 - \frac{s-1}{2})}
  \frac{\bf \zeta(\frac{s-2}{2})}{\bf \zeta(1 - \frac{s-2}{2})} )\\
-\log(\frac{{\bf (s-2)(s-1)}}{(2\pi i)^2} )
-({\bf 2s-3})\log 2 
-  \log(\frac{\sin(\frac{\pi (\frac{s+1}{2})}{2})}{\sin(\frac{\pi (1-\frac{s}{2})}{2})})
-\log(\frac{\sin(\frac{\pi (\frac{s-1}{2})}{2})}{\sin(\frac{\pi (1-\frac{s-2}{2})}{2})})
\end{split}
\end{equation}
implying

\begin{equation}
\begin{split}
\log(\frac{\zeta(\frac{s}{2})}{\zeta(1 -\frac{s}{2})} \frac{\zeta(\frac{s+1}{2})}{\zeta(1 - (\frac{s+1}{2}))} ) - (2s-1)\log(2)- 2\log(\frac{\sin(\frac{\pi (\frac{s+1}{2})}{2})}{\sin(\frac{\pi (1-\frac{s}{2})}{2})}) =\\
\log(  \frac{\zeta(\frac{s-1}{2})}{\zeta(1 - \frac{s-1}{2})}
  \frac{\zeta(\frac{s-2}{2})}{\zeta(1 - \frac{s-2}{2})} )
-\log(\frac{{\bf (s-2)(s-1)}}{(2\pi i)^2} )
-({\bf 2s-3})\log 2 
-  \log(\frac{\sin(\frac{\pi (\frac{s+1}{2})}{2})}{\sin(\frac{\pi (1-\frac{s}{2})}{2})})
-\log(\frac{\sin(\frac{\pi (\frac{s-1}{2})}{2})}{\sin(\frac{\pi (1-\frac{s-2}{2})}{2})})
\end{split}
\end{equation}

\begin{equation}\label{res}
\begin{split}
\log(\frac{\zeta(\frac{s+1}{2})}{\zeta(1 -\frac{s+1}{2})} \frac{\zeta(\frac{s}{2})}{\zeta(1 - (\frac{s}{2}))} ) =
\log(  \frac{\zeta(\frac{s-1}{2})}{\zeta(1 - \frac{s-1}{2})}
  \frac{\zeta(\frac{s-2}{2})}{\zeta(1 - \frac{s-2}{2})} )+
({\bf 2})\log 2 
-\log(\frac{{\bf (s-2)(s-1)}}{(2\pi i)^2} )+\\
+ \log(\frac{\sin(\frac{\pi (\frac{s+1}{2})}{2})}{\sin(\frac{\pi (1-\frac{s}{2})}{2})})
-\log(\frac{\sin(\frac{\pi (\frac{s-1}{2})}{2})}{\sin(\frac{\pi (1-\frac{s-2}{2})}{2})})
\end{split}
\end{equation}

Changing s such that $s \rightarrow s-2$ leads to: \\
\begin{equation}
\begin{split}
\log(\frac{\zeta(\frac{s-1}{2})}{\zeta(1 -\frac{s-1}{2})} \frac{\zeta(\frac{s-2}{2})}{\zeta(1 - (\frac{s-2}{2}))} ) =
\log(  \frac{\zeta(\frac{s-3}{2})}{\zeta(1 - \frac{s-3}{2})}
  \frac{\zeta(\frac{s-4}{2})}{\zeta(1 - \frac{s-4}{2})} )+\\
+({\bf 2})\log 2 
-\log(\frac{{\bf (s-4)(s-3)}}{(2\pi i)^2} )
+ \log(\frac{\sin(\frac{\pi (\frac{s-1}{2})}{2})}{\sin(\frac{\pi (1-\frac{s-2}{2})}{2})})
-\log(\frac{\sin(\frac{\pi (\frac{s-3}{2})}{2})}{\sin(\frac{\pi (1-\frac{s-4}{2})}{2})}).
\end{split}
\end{equation}

Substituting this last result into eq. \ref{res} produces

\begin{equation}
\begin{split}
\log(\frac{\zeta(\frac{s+1}{2})}{\zeta(1 -\frac{s+1}{2})} \frac{\zeta(\frac{s}{2})}{\zeta(1 - (\frac{s}{2}))} ) =
\biggr( 
\log(  \frac{\zeta(\frac{s-3}{2})}{\zeta(1 - \frac{s-3}{2})}
  \frac{\zeta(\frac{s-4}{2})}{\zeta(1 - \frac{s-4}{2})} )+\\
+({\bf 2})\log 2 
-\log(\frac{{\bf (s-4)(s-3)}}{(2\pi i)^2} )
+ \log(\frac{\sin(\frac{\pi (\frac{s-1}{2})}{2})}{\sin(\frac{\pi (1-\frac{s-2}{2})}{2})})
-\log(\frac{\sin(\frac{\pi (\frac{s-3}{2})}{2})}{\sin(\frac{\pi (1-\frac{s-4}{2})}{2})})
\biggr)\\
+({\bf 2})\log 2 
-\log(\frac{{\bf (s-2)(s-1)}}{(2\pi i)^2} )
+ \log(\frac{\sin(\frac{\pi (\frac{s+1}{2})}{2})}{\sin(\frac{\pi (1-\frac{s}{2})}{2})})
-\log(\frac{\sin(\frac{\pi (\frac{s-1}{2})}{2})}{\sin(\frac{\pi (1-\frac{s-2}{2})}{2})})
\end{split}
\end{equation}
and

\begin{equation}
\begin{split}
\log(\frac{\zeta(\frac{s+1}{2})}{\zeta(1 -\frac{s+1}{2})} \frac{\zeta(\frac{s}{2})}{\zeta(1 - (\frac{s}{2}))} ) =
\log(  \frac{\zeta(\frac{s-3}{2})}{\zeta(1 - \frac{s-3}{2})}
  \frac{\zeta(\frac{s-4}{2})}{\zeta(1 - \frac{s-4}{2})} )+\\
+({\bf 4})\log 2 
-\log(\frac{{\bf (s-1)(s-2)(s-3)(s-4)}}{(2\pi i)^4} )
+ \log(\frac{\sin(\frac{\pi (\frac{s+1}{2})}{2})}{\sin(\frac{\pi (1-\frac{s}{2})}{2})})
-\log(\frac{\sin(\frac{\pi (\frac{s-3}{2})}{2})}{\sin(\frac{\pi (1-\frac{s-4}{2})}{2})}).
\end{split}
\end{equation}
Taking an analytical look at the right hand term 

\begin{equation}
\begin{split}
\log(  \frac{\zeta(\frac{s-3}{2})}{\zeta(1 - \frac{s-3}{2})}
  \frac{\zeta(\frac{s-4}{2})}{\zeta(1 - \frac{s-4}{2})} ) =
\log( \frac{ \frac{s-3}{2} \frac{s-1}{2}}{(2\pi i)^2} \frac{\zeta(\frac{s+1}{2})}{\zeta(1 - \frac{s+1}{2})}
  \frac{ \frac{s-4}{2} \frac{s-2}{2}}{(2\pi i)^2} \frac{\zeta(\frac{s}{2})}{\zeta(1 - \frac{s}{2})} )
\end{split}
\end{equation}
implies

\begin{equation}
\begin{split}
\log(\frac{\zeta(\frac{s+1}{2})}{\zeta(1 -\frac{s+1}{2})} \frac{\zeta(\frac{s}{2})}{\zeta(1 - (\frac{s}{2}))} ) =
\log( \frac{ \frac{s-3}{2} \frac{s-1}{2}}{(2\pi i)^2} \frac{\zeta(\frac{s+1}{2})}{\zeta(1 - \frac{s+1}{2})}
  \frac{ \frac{s-4}{2} \frac{s-2}{2}}{(2\pi i)^2} \frac{\zeta(\frac{s}{2})}{\zeta(1 - \frac{s}{2})} )+\\
+({\bf 4})\log 2 
-\log(\frac{{\bf (s-1)(s-2)(s-3)(s-4)}}{(2\pi i)^4} )
+ \log(\frac{\sin(\frac{\pi (\frac{s+1}{2})}{2})}{\sin(\frac{\pi (1-\frac{s}{2})}{2})})
-\log(\frac{\sin(\frac{\pi (\frac{s-3}{2})}{2})}{\sin(\frac{\pi (1-\frac{s-4}{2})}{2})})
\end{split}
\end{equation}
and as a result of further simplification, 

\begin{equation}\label{sine}
\begin{split}
\log(\frac{\sin(\frac{\pi (\frac{s+1}{2})}{2})}{\sin(\frac{\pi (1-\frac{s}{2})}{2})})
=\log(\frac{\sin(\frac{\pi (\frac{ s+1 \bf-4}{2})}{2})}{\sin(\frac{\pi (1-\frac{s\bf -4}{2})}{2})});
\end{split}
\end{equation}
confirming

\begin{equation}
\begin{split}
\log(\frac{\sin(\frac{\pi}{2} ({\bf \frac{s}{2}}+\frac{1}{2}))}{\cos(\frac{\pi}{2} ({\bf \frac{s}{2}}))})
=\log(\frac{\sin(\frac{\pi}{2} ({\bf \frac{s-4}{2}}+\frac{1}{2}))}{\cos(\frac{\pi}{2} ({\bf \frac{s-4}{2}}))}).
\end{split}
\end{equation}

\end{proof}

%******************************************************************************************
\section{Further analysis involving the gamma function}
Here, we first prove the following theorem before presenting the related identity in terms of zeta function.
\begin{theorem}
\textcolor{blue}{
\begin{equation}
\begin{split}
\log(\frac{\Gamma(2+\frac{s}{2})}{\Gamma(\frac{1}{2}-2-\frac{s}{2})})
-\log(\frac{\Gamma(2-\frac{s}{2})}{\Gamma(\frac{1}{2}-2+\frac{s}{2})})
=
\log(\frac{\Gamma(\frac{s}{2}-2)}{\Gamma(\frac{1}{2}+2-\frac{s}{2})})
-\log(\frac{\Gamma(-\frac{s}{2}-2)}{\Gamma(\frac{1}{2}+2+\frac{s}{2})})
+\log( \frac{s-4}{s+4})
\end{split}
\end{equation}
}
\end{theorem}

\begin{proof}
Due to theorem \ref{logzeta}:

%\textcolor{blue}{
%\begin{equation}
%\begin{split}
%\log(\frac{\Gamma^2(-\frac{s}{2})}{\Gamma^2(\frac{1+s}{2})}) +
%\log(\pi^{2(1+s-\frac{1}{2})})=
%-\log( \frac{\Gamma^2(\frac{s}{2})}{\Gamma^2(\frac{1-s}{2})})
%-\log(\pi^{2(1-s-\frac{1}{2})})
%-\log( \frac{2{\bf s^2}}{\bf (s-2)} \frac{{\bf (\frac{s}{2}-1)}}{(2\pi i)^2} )+\\
%- \log(\frac{\sin(\frac{\pi (\frac{s+2}{2})}{2})}{\sin(\frac{\pi (1-\frac{s+1}{2})}{2})})
%+  \log(\frac{\sin(\frac{\pi (\frac{s+1}{2})}{2})}{\sin(\frac{\pi (1-\frac{s}{2})}{2})})
%-\log(\frac{\sin(\frac{\pi (\frac{s}{2})}{2})}{\sin(\frac{\pi (1-\frac{s-1}{2})}{2})})
%+\log(\frac{\sin(\frac{\pi (\frac{s-1}{2})}{2})}{\sin(\frac{\pi (1-\frac{s-2}{2})}{2})}).
%\end{split}
%\end{equation}
%}

\begin{equation}\label{gam1}
\begin{split}
\log(\frac{\Gamma^2(-\frac{s}{2})}{\Gamma^2(\frac{1+s}{2})}) +
\log(\pi^{2(1+s-\frac{1}{2})})=
-\log( \frac{\Gamma^2(\frac{s}{2})}{\Gamma^2(\frac{1-s}{2})})
-\log(\pi^{2(1-s-\frac{1}{2})})
-\log( \frac{2{\bf s^2}}{\bf (s-2)} \frac{{\bf (\frac{s}{2}-1)}}{(2\pi i)^2} )+\\
\textcolor{blue}{
- \log(\frac{\sin(\frac{\pi (\frac{s+2}{2})}{2})}{\sin(\frac{\pi (1-\frac{s+1}{2})}{2})})
+  \log(\frac{\sin(\frac{\pi (\frac{s+1}{2})}{2})}{\sin(\frac{\pi (1-\frac{s}{2})}{2})})
-\log(\frac{\sin(\frac{\pi (\frac{s}{2})}{2})}{\sin(\frac{\pi (1-\frac{s-1}{2})}{2})})
+\log(\frac{\sin(\frac{\pi (\frac{s-1}{2})}{2})}{\sin(\frac{\pi (1-\frac{s-2}{2})}{2})}).
}
\end{split}
\end{equation}
Also

\begin{equation}\label{gam2}
\begin{split}
\log(\frac{\Gamma^2(-\frac{(s-4)}{2})}{\Gamma^2(\frac{1+s-4}{2})}) +
\log(\pi^{2(1+s-4-\frac{1}{2})}) \\
=
-\log( \frac{\Gamma^2(\frac{s-4}{2})}{\Gamma^2(\frac{1-(s-4)}{2})})
-\log(\pi^{2(1-(s-4)-\frac{1}{2})})
-\log( \frac{2{\bf (s-4)^2}}{\bf (s-4-2)} \frac{{\bf (\frac{s-4}{2}-1)}}{(2\pi i)^2} )+\\
\textcolor{blue}{
- \log(\frac{\sin(\frac{\pi (\frac{s+2}{2})}{2})}{\sin(\frac{\pi (1-\frac{s+1}{2})}{2})})
+  \log(\frac{\sin(\frac{\pi (\frac{s+1}{2})}{2})}{\sin(\frac{\pi (1-\frac{s}{2})}{2})})
-\log(\frac{\sin(\frac{\pi (\frac{s}{2})}{2})}{\sin(\frac{\pi (1-\frac{s-1}{2})}{2})})
+\log(\frac{\sin(\frac{\pi (\frac{s-1}{2})}{2})}{\sin(\frac{\pi (1-\frac{s-2}{2})}{2})}),
}
\end{split}
\end{equation}
because of the following identities which are derived in relation to the identity \ref{sine}:\\
\[
\log(\frac{\sin(\frac{\pi (\frac{s+2}{2})}{2})}{\sin(\frac{\pi (1-\frac{s+1}{2})}{2})})
=\log(\frac{\sin(\frac{\pi (\frac{ {\bf s-4}+2}{2})}{2})}{\sin(\frac{\pi (1-\frac{{\bf s-4}+1}{2})}{2})});
\]
\[
\log(\frac{\sin(\frac{\pi (\frac{s+1}{2})}{2})}{\sin(\frac{\pi (1-\frac{s}{2})}{2})})
=\log(\frac{\sin(\frac{\pi (\frac{ {\bf s-4}+1}{2})}{2})}{\sin(\frac{\pi (1-\frac{{\bf s-4}}{2})}{2})})
;
\]
 \[
\vdots
\]
and so on. Subtracting eq. \ref{gam2} from \ref{gam1}:
\begin{equation}
\begin{split}
\log(\frac{\Gamma^2(-\frac{s}{2})}{\Gamma^2(\frac{1+s}{2})}) +
\log(\pi^{2(1+s-\frac{1}{2})}) 
- \biggr( \log(\frac{\Gamma^2(-\frac{(s-4)}{2})}{\Gamma^2(\frac{1+s-4}{2})}) +\log(\pi^{2(1+s-4-\frac{1}{2})}) \biggr) =\\
-\log( \frac{\Gamma^2(\frac{s}{2})}{\Gamma^2(\frac{1-s}{2})})
-\log(\pi^{2(1-s-\frac{1}{2})})
-\log( \frac{2{\bf s^2}}{\bf (s-2)} \frac{{\bf (\frac{s}{2}-1)}}{(2\pi i)^2} )\\
-\biggr( -\log( \frac{\Gamma^2(\frac{s-4}{2})}{\Gamma^2(\frac{1-(s-4)}{2})})
-\log(\pi^{2(1-(s-4)-\frac{1}{2})})
-\log( \frac{2{\bf (s-4)^2}}{\bf (s-4-2)} \frac{{\bf (\frac{s-4}{2}-1)}}{(2\pi i)^2} ) \biggr)
\end{split}
\end{equation}

$\Rightarrow$

\begin{equation}
\begin{split}
\log(\frac{\Gamma^2(-\frac{s}{2})}{\Gamma^2(\frac{1+s}{2})}) 
- \log(\frac{\Gamma^2(2-\frac{s}{2})}{\Gamma^2(\frac{1+s}{2}-2)})  = 
-\log( \frac{\Gamma^2(\frac{s}{2})}{\Gamma^2(\frac{1-s}{2})})
-\log( \frac{{\bf s^2}}{\bf (s-2)} {\bf (\frac{s}{2}-1)} )+\\
+\log( \frac{\Gamma^2(\frac{s}{2}-2)}{\Gamma^2(2+\frac{1-s}{2})})
+\log( \frac{{\bf (s-4)^2}}{\bf (s-6)} {\bf (\frac{s}{2}-3)} )
\end{split}
\end{equation}

$\Rightarrow$

\begin{equation}
\begin{split}
\log(\frac{\Gamma^2(-\frac{s}{2})}{\Gamma^2(\frac{1+s}{2})})+\log( \frac{\Gamma^2(\frac{s}{2})}{\Gamma^2(\frac{1-s}{2})}) +\log( \frac{{\bf s^2}}{\bf (s-2)} {\bf (\frac{s}{2}-1)} ) \\= 
\log(\frac{\Gamma^2(2-\frac{s}{2})}{\Gamma^2(\frac{1+s}{2}-2)})  
+\log( \frac{\Gamma^2(\frac{s}{2}-2)}{\Gamma^2(2+\frac{1-s}{2})})
+\log( \frac{{\bf (s-4)^2}}{\bf (s-6)} {\bf (\frac{s}{2}-3)} )
\end{split}
\end{equation}

$\Rightarrow$

\begin{equation}
\begin{split}
\log(\frac{\Gamma^2(-\frac{s}{2})}{\Gamma^2(\frac{1+s}{2})})+\log( \frac{\Gamma^2(\frac{s}{2})}{\Gamma^2(\frac{1-s}{2})}) +\log( \frac{{\bf s^2}}{\bf 2} ) = 
\log(\frac{\Gamma^2(2-\frac{s}{2})}{\Gamma^2(\frac{1+s}{2}-2)})  
+\log( \frac{\Gamma^2(\frac{s}{2}-2)}{\Gamma^2(2+\frac{1-s}{2})})
+\log( \frac{{\bf (s-4)^2}}{\bf 2})
\end{split}
\end{equation}

$\Rightarrow$

\begin{equation}
\begin{split}
\log(\frac{\Gamma^2(-\frac{s}{2})}{\Gamma^2(\frac{1+s}{2})})+\log( \frac{\Gamma^2(\frac{s}{2})}{\Gamma^2(\frac{1-s}{2})}) +\log( s^2 ) = 
\log(\frac{\Gamma^2(2-\frac{s}{2})}{\Gamma^2(\frac{1+s}{2}-2)})  
+\log( \frac{\Gamma^2(\frac{s}{2}-2)}{\Gamma^2(2+\frac{1-s}{2})})
+\log( (s-4)^2)
\end{split}
\end{equation}

$\Rightarrow$

\begin{equation}\label{temp1}
\begin{split}
\log(\frac{\Gamma(-\frac{s}{2})}{\Gamma(\frac{1}{2}+\frac{s}{2})})
+\log(\frac{\Gamma(\frac{s}{2})}{\Gamma(\frac{1}{2}-\frac{s}{2})}) 
+\log( s) 
= 
\log(\frac{\Gamma(2-\frac{s}{2})}{\Gamma(\frac{1}{2}-2+\frac{s}{2})})
+\log(\frac{\Gamma(\frac{s}{2}-2)}{\Gamma(\frac{1}{2}+2-\frac{s}{2})})
+\log( s-4)
\end{split}
\end{equation}

$s \rightarrow -s$:
\begin{equation}\label{temp2}
\begin{split}
\log(\frac{\Gamma(-\frac{s}{2})}{\Gamma(\frac{1}{2}+\frac{s}{2})})
+\log(\frac{\Gamma(\frac{s}{2})}{\Gamma(\frac{1}{2}-\frac{s}{2})}) 
+\log( -s) 
= 
\log(\frac{\Gamma(2+\frac{s}{2})}{\Gamma(\frac{1}{2}-2-\frac{s}{2})})
+\log(\frac{\Gamma(-\frac{s}{2}-2)}{\Gamma(\frac{1}{2}+2+\frac{s}{2})})
+\log( -s-4).
\end{split}
\end{equation}

Subtracting eq. \ref{temp2} from \ref{temp1}:

\begin{equation}
\begin{split}
\log(\frac{\Gamma(2+\frac{s}{2})}{\Gamma(\frac{1}{2}-2-\frac{s}{2})})
+\log(\frac{\Gamma(-\frac{s}{2}-2)}{\Gamma(\frac{1}{2}+2+\frac{s}{2})})
= 
\log(\frac{\Gamma(2-\frac{s}{2})}{\Gamma(\frac{1}{2}-2+\frac{s}{2})})
+\log(\frac{\Gamma(\frac{s}{2}-2)}{\Gamma(\frac{1}{2}+2-\frac{s}{2})})
+\log( \frac{s-4}{s+4})
\end{split}
\end{equation}

Therefore

\[
\begin{split}
\log(\frac{\Gamma(2+\frac{s}{2})}{\Gamma(\frac{1}{2}-2-\frac{s}{2})})
-\log(\frac{\Gamma(2-\frac{s}{2})}{\Gamma(\frac{1}{2}-2+\frac{s}{2})})
=
\log(\frac{\Gamma(\frac{s}{2}-2)}{\Gamma(\frac{1}{2}+2-\frac{s}{2})})
-\log(\frac{\Gamma(-\frac{s}{2}-2)}{\Gamma(\frac{1}{2}+2+\frac{s}{2})})
+\log( \frac{s-4}{s+4})
\end{split}
\]

or

\[
\begin{split}
-\log(\frac{\Gamma(2+\frac{s}{2})}{\Gamma(\frac{1}{2}-2-\frac{s}{2})})
+\log(\frac{\Gamma(2-\frac{s}{2})}{\Gamma(\frac{1}{2}-2+\frac{s}{2})})
=
-\log(\frac{\Gamma(\frac{s}{2}-2)}{\Gamma(\frac{1}{2}+2-\frac{s}{2})})
+\log(\frac{\Gamma(-\frac{s}{2}-2)}{\Gamma(\frac{1}{2}+2+\frac{s}{2})})
+\log( \frac{s+4}{s-4})
\end{split}
\]

as a result of  $s \rightarrow -s$ substitution. 
\end{proof}

%**********************************************************************************
\subsection{Logarithm of zeta at $\pm s \pm 4$ and $1-(\pm s \pm 4)$}
The related zeta version of the last equation may be presented in form of 
\begin{equation}
\begin{split}
-\log(\frac{\zeta(s+4)}{\zeta(-3-s)}\pi^{\frac{1}{2}-(s+4)})
+\log(\frac{\zeta(-s+4)}{\zeta(-3+s)}\pi^{\frac{1}{2}-(-s+4)}) \\
=
-\log(\frac{\zeta(s-4)}{\zeta(5-s)}\pi^{\frac{1}{2}-(s-4)})
+\log(\frac{\zeta(-s-4)}{\zeta(5+s)}\pi^{\frac{1}{2}-(-s-4)})
+\log( \frac{s-4}{s+4})
\end{split}
\end{equation}

$\Rightarrow$

\begin{equation}
\begin{split}
-\log(\frac{\zeta(s+4)}{\zeta(-3-s)})
+\log(\frac{\zeta(-s+4)}{\zeta(-3+s)})
+\log(\frac{\pi^{\frac{1}{2}-(-s+4)}}{\pi^{\frac{1}{2}-(s+4)}})\\
=
-\log(\frac{\zeta(s-4)}{\zeta(5-s)})
+\log(\frac{\zeta(-s-4)}{\zeta(5+s)})
+\log(\frac{\pi^{\frac{1}{2}-(-s-4)}}{\pi^{\frac{1}{2}-(s-4)}})
+\log( \frac{s-4}{s+4})
\end{split}
\end{equation}

$\Rightarrow$

\begin{equation}
\begin{split}
-\log(\frac{\zeta(s+4)}{\zeta(-3-s)})
+\log(\frac{\zeta(-s+4)}{\zeta(-3+s)})
+\log(\pi^{2s})
=
-\log(\frac{\zeta(s-4)}{\zeta(5-s)})
+\log(\frac{\zeta(-s-4)}{\zeta(5+s)})
+\log(\pi^{2s})
+\log( \frac{s-4}{s+4})
\end{split}
\end{equation}

$\Rightarrow$

\textcolor{blue}{
\begin{equation}\label{zet4}
\begin{split}
-\log(\frac{\zeta(s+4)}{\zeta(-3-s)})
+\log(\frac{\zeta(-s+4)}{\zeta(-3+s)})
=
-\log(\frac{\zeta(s-4)}{\zeta(5-s)})
+\log(\frac{\zeta(-s-4)}{\zeta(5+s)})
+\log( \frac{s-4}{s+4}).
\end{split}
\end{equation}
}

This last statement is not difficult to justify. To this aim, we use the recursion relation substitutes as demonstrated:
\begin{equation}
\begin{split}
-\log(\frac{\zeta(s+4)}{\zeta(-3-s)})
+\log(\frac{\zeta(-s+4)}{\zeta(-3+s)})
=
-\log(\frac{\zeta(s-4)}{\zeta(5-s)})
+\log(\frac{\zeta(-s-4)}{\zeta(5+s)})
+\log( \frac{s-4}{s+4})
\\
=
-\log( \frac{(s-4)(s-3) (s-2)(s-1)(s)(s+1) (s+2)(s+3)}{(2 \pi i)^4} \frac{\zeta(s+4)}{\zeta(-3-s)}) 
\\
+\log(\frac{(-s-4)(-s-3) (-s-2)(-s-1)(-s)(-s+1) (-s+2)(-s+3)}{(2 \pi i)^4} \frac{\zeta(-s+4)}{\zeta(-3+s)})
\\
+\log( \frac{s-4}{s+4})
\end{split}
\end{equation}
which then confirms the obvious (but almost unreported) identity 

\begin{equation}
\begin{split}
\log( \frac{(s-4)(s-3) (s-2)(s-1)(s)(s+1) (s+2)(s+3)}{(2 \pi i)^4} ) 
\\
-\log(\frac{(-s-4)(-s-3) (-s-2)(-s-1)(-s)(-s+1) (-s+2)(-s+3)}{(2 \pi i)^4} )
\\
=\log( \frac{s-4}{s+4}).
\end{split}
\end{equation}
presented here for the first time ever.

%********************************************************************************************

\subsection{Logarithm of zeta at $\pm s$ and $1\mp s$}
Another important discovery involving eq. \ref{zet4} is on the proof of the next theorem.

\begin{theorem}
\textcolor{blue}{
\small
\begin{equation}
\begin{split}
\log(\frac{\zeta(s)}{\zeta(1-s)})
+\log(\frac{\zeta(-s)}{\zeta(1+s)})
\\
=
\log(\frac{\zeta(-s+8)}{\zeta(-7+s)})
+\log(\frac{\zeta(s-8)}{\zeta(9-s)})
-\log( \frac{s-8}{s})
\\
\textcolor{black}{
=
\log(\frac{\zeta(-s+16)}{\zeta(-15+s)})
+\log(\frac{\zeta(s-16)}{\zeta(17-s)})
-\log( \frac{s-16}{s-8})
+\log( \frac{-s}{-s+8})
}
\\
=
\log(\frac{\zeta(s+8)}{\zeta(-7-s)})
+\log(\frac{\zeta(-s-8)}{\zeta(9+s)})
-\log( \frac{-s-8}{-s})
\\
\textcolor{black}{
=
\log(\frac{\zeta(s+16)}{\zeta(-15-s)})
+\log(\frac{\zeta(-s-16)}{\zeta(17+s)})
-\log( \frac{-s-16}{-s-8})
+\log( \frac{s}{s+8})
}
\end{split}
\end{equation}
}
\end{theorem}

\begin{proof}
The fact that 
\[
\begin{split}
-\log(\frac{\zeta(s+4)}{\zeta(-3-s)})
+\log(\frac{\zeta(-s+4)}{\zeta(-3+s)})
=
-\log(\frac{\zeta(s-4)}{\zeta(5-s)})
+\log(\frac{\zeta(-s-4)}{\zeta(5+s)})
+\log( \frac{s-4}{s+4})
\end{split}
\]
suggests

\begin{equation}
\begin{split}
-\log(\frac{\zeta(s)}{\zeta(1-s)})
+\log(\frac{\zeta(-s+8)}{\zeta(-7+s)})
=
-\log(\frac{\zeta(s-8)}{\zeta(9-s)})
+\log(\frac{\zeta(-s)}{\zeta(1+s)})
+\log( \frac{s-8}{s})
\end{split}
\end{equation}
$\Rightarrow$

\begin{equation}
\begin{split}
-\log(\frac{\zeta(-s)}{\zeta(1+s)})
-\log(\frac{\zeta(s)}{\zeta(1-s)})
=
-\log(\frac{\zeta(-s+8)}{\zeta(-7+s)})
-\log(\frac{\zeta(s-8)}{\zeta(9-s)})
+\log( \frac{s-8}{s})
\end{split}
\end{equation}
$\Rightarrow$

\begin{equation}
\begin{split}
\log(\frac{\zeta(-s)}{\zeta(1+s)})
+\log(\frac{\zeta(s)}{\zeta(1-s)})
=
\log(\frac{\zeta(-s+8)}{\zeta(-7+s)})
+\log(\frac{\zeta(s-8)}{\zeta(9-s)})
-\log( \frac{s-8}{s})
\end{split}
\end{equation}
This implies

\textcolor{blue}{
\begin{equation}\label{zetap}
\begin{split}
\log(\frac{\zeta(s)}{\zeta(1-s)})
+\log(\frac{\zeta(-s)}{\zeta(1+s)})
=
\log(\frac{\zeta(-s+8)}{\zeta(-7+s)})
+\log(\frac{\zeta(s-8)}{\zeta(9-s)})
-\log( \frac{s-8}{s})
\\
=
\log(\frac{\zeta(s+8)}{\zeta(-7-s)})
+\log(\frac{\zeta(-s-8)}{\zeta(9+s)})
-\log( \frac{-s-8}{-s}).
\end{split}
\end{equation}
}
Substituting $s \rightarrow -s+8$ into \ref{zetap}:

\begin{equation}
\begin{split}
\log(\frac{\zeta(-s+8)}{\zeta(-7+s)})
+\log(\frac{\zeta(s-8)}{\zeta(9-s)})
=
\log(\frac{\zeta(s)}{\zeta(1-s)})
+\log(\frac{\zeta(-s)}{\zeta(1+s)})
-\log( \frac{-s}{-s+8})
\\
=
\log(\frac{\zeta(-s+16)}{\zeta(-15+s)})
+\log(\frac{\zeta(s-16)}{\zeta(17-s)})
-\log( \frac{s-16}{s-8}).
\end{split}
\end{equation}
Substituting $s \rightarrow s+8$ into \ref{zetap}:

\begin{equation}
\begin{split}
\log(\frac{\zeta(s+8)}{\zeta(-7-s)})
+\log(\frac{\zeta(-s-8)}{\zeta(9+s)})
=
\log(\frac{\zeta(-s)}{\zeta(1+s)})
+\log(\frac{\zeta(s)}{\zeta(1-s)})
-\log( \frac{s}{s+8})
\\
=
\log(\frac{\zeta(s+16)}{\zeta(-15-s)})
+\log(\frac{\zeta(-s-16)}{\zeta(17+s)})
-\log( \frac{-s-16}{-s-8}).
\end{split}
\end{equation}
In other words,
{
\small
\[
\begin{split}
\log(\frac{\zeta(s)}{\zeta(1-s)})
+\log(\frac{\zeta(-s)}{\zeta(1+s)})
\\
=
\log(\frac{\zeta(-s+8)}{\zeta(-7+s)})
+\log(\frac{\zeta(s-8)}{\zeta(9-s)})
-\log( \frac{s-8}{s})
\\
\textcolor{black}{
=
\log(\frac{\zeta(-s+16)}{\zeta(-15+s)})
+\log(\frac{\zeta(s-16)}{\zeta(17-s)})
-\log( \frac{s-16}{s-8})
+\log( \frac{-s}{-s+8})
}
\\
=
\log(\frac{\zeta(s+8)}{\zeta(-7-s)})
+\log(\frac{\zeta(-s-8)}{\zeta(9+s)})
-\log( \frac{-s-8}{-s})
\\
\textcolor{black}{
=
\log(\frac{\zeta(s+16)}{\zeta(-15-s)})
+\log(\frac{\zeta(-s-16)}{\zeta(17+s)})
-\log( \frac{-s-16}{-s-8})
+\log( \frac{s}{s+8})
}
\end{split}
\]
}
\end{proof}

%*************************************************************************************************************

\section{An introduction to the new omega function $\Omega(s)$}
Here, we wish to present and prove 
\begin{theorem}
\textcolor{blue}{
\begin{equation}
\begin{split}
\Omega(s) = 
\psi^{(2n)}(\frac{s}{2}) + \psi^{(2n)}(\frac{1-s}{2}) - 
\psi^{(2n)}(1-\frac{s}{2}) - \psi^{(2n)}(\frac{1+s}{2})
=\\
\biggr( \psi^{(2n)}(\frac{s}{2}) - \psi^{(2n)}(\frac{1}{2}+\frac{s}{2}) \biggr)+ 
\biggr(\psi^{(2n)}(\frac{1-s}{2}) - 
 \psi^{(2n)}(\frac{1}{2}+\frac{1-s}{2}) \biggr)
=\\
(-2)^{2n+1}\Gamma(2n+1)\biggr(\sum^{\infty}_{k=0}{ \frac{(-1)^k} {(k+s)^{2n+1} } } +
\sum^{\infty}_{k=0}{ \frac{(-1)^k} {(k+1-s)^{2n+1} } } \biggr) =\\
2^{2n}\frac{d^{2n+1}}{ds^{2n+1}} \biggr(
- \log(\frac{\sin(\frac{\pi (\frac{s+2}{2})}{2})}{\cos(\frac{\pi (\frac{s+1}{2})}{2})})
+  \log(\frac{\sin(\frac{\pi (\frac{s+1}{2})}{2})}{\cos(\frac{\pi (\frac{s}{2})}{2})})
-\log(\frac{\sin(\frac{\pi (\frac{s}{2})}{2})}{\cos(\frac{\pi (\frac{s-1}{2})}{2})})
+\log(\frac{\sin(\frac{\pi (\frac{s-1}{2})}{2})}{\cos(\frac{\pi (\frac{s-2}{2})}{2})})
\biggr)
= \\
-\Gamma(2n+1) \biggr(
\sum^{\infty}_{k=0}{ \frac{1} {(k+\frac{s}{2})^{2n+1} } }+
\sum^{\infty}_{k=0}{ \frac{1} {(k+\frac{1-s}{2})^{2n+1} } }-
\sum^{\infty}_{k=0}{ \frac{1} {(k+1-\frac{s}{2})^{2n+1} } }-
\sum^{\infty}_{k=0}{ \frac{1} {(k+\frac{1+s}{2})^{2n+1} } }
\biggr)
\end{split}
\end{equation}
}
\end{theorem}

\begin{proof}
According to theorem \ref{logzeta}:
 
\begin{equation}
\begin{split}
\log(\frac{\Gamma^2(-\frac{s}{2})}{\Gamma^2(\frac{1+s}{2})}) +
\log(\pi^{2(1+s-\frac{1}{2})})=
-\log( \frac{\Gamma^2(\frac{s}{2})}{\Gamma^2(\frac{1-s}{2})})
-\log(\pi^{2(1-s-\frac{1}{2})})
-\log( \frac{2{\bf s^2}}{\bf (s-2)} \frac{{\bf (\frac{s}{2}-1)}}{(2\pi i)^2} )+\\
- \log(\frac{\sin(\frac{\pi (\frac{s+2}{2})}{2})}{\sin(\frac{\pi (1-\frac{s+1}{2})}{2})})
+  \log(\frac{\sin(\frac{\pi (\frac{s+1}{2})}{2})}{\sin(\frac{\pi (1-\frac{s}{2})}{2})})
-\log(\frac{\sin(\frac{\pi (\frac{s}{2})}{2})}{\sin(\frac{\pi (1-\frac{s-1}{2})}{2})})
+\log(\frac{\sin(\frac{\pi (\frac{s-1}{2})}{2})}{\sin(\frac{\pi (1-\frac{s-2}{2})}{2})})
\end{split}
\end{equation}
$\Rightarrow$

\[
\begin{split}
\frac{d^{2n+1}}{ds^{2n+1}}(
\log(\frac{\Gamma^2(-\frac{s}{2})}{\Gamma^2(\frac{1+s}{2})})
) =
-
\frac{d^{2n+1}}{ds^{2n+1}}(
\log( \frac{\Gamma^2(\frac{s}{2})}{\Gamma^2(\frac{1-s}{2})})
)
-
\frac{d^{2n+1}}{ds^{2n+1}} (
\log( \frac{2{\bf s^2}}{\bf (s-2)} \frac{{\bf (\frac{s}{2}-1)}}{(2\pi i)^2} )
)+\\
\frac{d^{2n+1}}{ds^{2n+1}} \biggr(
- \log(\frac{\sin(\frac{\pi (\frac{s+2}{2})}{2})}{\sin(\frac{\pi (1-\frac{s+1}{2})}{2})})
+  \log(\frac{\sin(\frac{\pi (\frac{s+1}{2})}{2})}{\sin(\frac{\pi (1-\frac{s}{2})}{2})})
-\log(\frac{\sin(\frac{\pi (\frac{s}{2})}{2})}{\sin(\frac{\pi (1-\frac{s-1}{2})}{2})})
+\log(\frac{\sin(\frac{\pi (\frac{s-1}{2})}{2})}{\sin(\frac{\pi (1-\frac{s-2}{2})}{2})})
\biggr)
\end{split}
\]
$\Rightarrow$

\textcolor{blue}{
\begin{equation}
\begin{split}
- \frac{\psi^{(2)}(-\frac{s}{2}) + \psi^{(2)}(\frac{1+s}{2})}{2^{2n}} =
-
\frac{\psi^{(2)}(\frac{s}{2}) + \psi^{(2)}(\frac{1-s}{2})}{2^{2n}}
-
\frac{d^{2n+1}}{ds^{2n+1}} (
\log( s^2 )
)+\\
\frac{d^{2n+1}}{ds^{2n+1}} \biggr(
- \log(\frac{\sin(\frac{\pi (\frac{s+2}{2})}{2})}{\sin(\frac{\pi (1-\frac{s+1}{2})}{2})})
+  \log(\frac{\sin(\frac{\pi (\frac{s+1}{2})}{2})}{\sin(\frac{\pi (1-\frac{s}{2})}{2})})
-\log(\frac{\sin(\frac{\pi (\frac{s}{2})}{2})}{\sin(\frac{\pi (1-\frac{s-1}{2})}{2})})
+\log(\frac{\sin(\frac{\pi (\frac{s-1}{2})}{2})}{\sin(\frac{\pi (1-\frac{s-2}{2})}{2})})
\biggr)
\end{split}
\end{equation}
}

Note the following identities:
\[ \psi^{(2n)}({a \over b}) = {(-1)^{2n+1}\Gamma(2n+1) \over ({a \over b})^{2n+1}}+\psi^{(2n)}(1+{a \over b});  \]
\[ \psi^{(2n)}(1+{a \over b}) =  {(-1)^{2n}\Gamma(2n+1) \over ({a \over b})^{2n+1}}+\psi^{(2n)}({a \over b});  \]
i.e.
\[ \psi^{(2n)}(-{s \over 2}) = -{\Gamma(2n+1) \over (-{s \over 2})^{2n+1}}+\psi^{(2n)}(1-{s \over 2}).  \]
Therefore

\begin{equation}
\begin{split}
- \frac{\psi^{(2)}(1-\frac{s}{2}) + \psi^{(2)}(\frac{1+s}{2})}{2^{2n}}=
-
\frac{\psi^{(2)}(\frac{s}{2}) + \psi^{(2)}(\frac{1-s}{2})}{2^{2n}}
- \frac{\Gamma(2n+1)}{2^{2n}(-{s \over 2})^{2n+1}}
-
\frac{d^{2n+1}}{ds^{2n+1}} (
\log( s^2 )
)+\\
\frac{d^{2n+1}}{ds^{2n+1}} \biggr(
- \log(\frac{\sin(\frac{\pi (\frac{s+2}{2})}{2})}{\sin(\frac{\pi (1-\frac{s+1}{2})}{2})})
+  \log(\frac{\sin(\frac{\pi (\frac{s+1}{2})}{2})}{\sin(\frac{\pi (1-\frac{s}{2})}{2})})
-\log(\frac{\sin(\frac{\pi (\frac{s}{2})}{2})}{\sin(\frac{\pi (1-\frac{s-1}{2})}{2})})
+\log(\frac{\sin(\frac{\pi (\frac{s-1}{2})}{2})}{\sin(\frac{\pi (1-\frac{s-2}{2})}{2})})
\biggr)
\end{split}
\end{equation}
$\Rightarrow$

\begin{equation}
\begin{split}
- \frac{\psi^{(2)}(1-\frac{s}{2}) + \psi^{(2)}(\frac{1+s}{2})}{2^{2n}}=
-
\frac{\psi^{(2)}(\frac{s}{2}) + \psi^{(2)}(\frac{1-s}{2})}{2^{2n}}
+ \frac{2\Gamma(2n+1)}{s^{2n+1}}
-
\frac{d^{2n+1}}{ds^{2n+1}} (
\log( s^2 )
)+\\
\frac{d^{2n+1}}{ds^{2n+1}} \biggr(
- \log(\frac{\sin(\frac{\pi (\frac{s+2}{2})}{2})}{\sin(\frac{\pi (1-\frac{s+1}{2})}{2})})
+  \log(\frac{\sin(\frac{\pi (\frac{s+1}{2})}{2})}{\sin(\frac{\pi (1-\frac{s}{2})}{2})})
-\log(\frac{\sin(\frac{\pi (\frac{s}{2})}{2})}{\sin(\frac{\pi (1-\frac{s-1}{2})}{2})})
+\log(\frac{\sin(\frac{\pi (\frac{s-1}{2})}{2})}{\sin(\frac{\pi (1-\frac{s-2}{2})}{2})})
\biggr)
\end{split}
\end{equation}
$\Rightarrow$

\begin{equation}
\begin{split}
- \frac{\psi^{(2)}(1-\frac{s}{2}) + \psi^{(2)}(\frac{1+s}{2})}{2^{2n}}=
-
\frac{\psi^{(2)}(\frac{s}{2}) + \psi^{(2)}(\frac{1-s}{2})}{2^{2n}}
+\\
+\frac{d^{2n+1}}{ds^{2n+1}} \biggr(
- \log(\frac{\sin(\frac{\pi (\frac{s+2}{2})}{2})}{\sin(\frac{\pi (1-\frac{s+1}{2})}{2})})
+  \log(\frac{\sin(\frac{\pi (\frac{s+1}{2})}{2})}{\sin(\frac{\pi (1-\frac{s}{2})}{2})})
-\log(\frac{\sin(\frac{\pi (\frac{s}{2})}{2})}{\sin(\frac{\pi (1-\frac{s-1}{2})}{2})})
+\log(\frac{\sin(\frac{\pi (\frac{s-1}{2})}{2})}{\sin(\frac{\pi (1-\frac{s-2}{2})}{2})})
\biggr)
\end{split}
\end{equation}
$\Rightarrow$

\begin{equation}
\begin{split}
\frac{\psi^{(2)}(\frac{s}{2}) + \psi^{(2)}(\frac{1-s}{2})}{2^{2n}} - \frac{\psi^{(2)}(1-\frac{s}{2}) + \psi^{(2)}(\frac{1+s}{2})}{2^{2n}} 
=\\
\frac{d^{2n+1}}{ds^{2n+1}} \biggr(
- \log(\frac{\sin(\frac{\pi (\frac{s+2}{2})}{2})}{\sin(\frac{\pi (1-\frac{s+1}{2})}{2})})
+  \log(\frac{\sin(\frac{\pi (\frac{s+1}{2})}{2})}{\sin(\frac{\pi (1-\frac{s}{2})}{2})})
-\log(\frac{\sin(\frac{\pi (\frac{s}{2})}{2})}{\sin(\frac{\pi (1-\frac{s-1}{2})}{2})})
+\log(\frac{\sin(\frac{\pi (\frac{s-1}{2})}{2})}{\sin(\frac{\pi (1-\frac{s-2}{2})}{2})})
\biggr)
= \\
-\frac{\Gamma(2n+1)}{2^{2n}} \biggr(
\sum^{\infty}_{k=0}{ \frac{1} {(k+\frac{s}{2})^{2n+1} } }+
\sum^{\infty}_{k=0}{ \frac{1} {(k+\frac{1-s}{2})^{2n+1} } }-
\sum^{\infty}_{k=0}{ \frac{1} {(k+1-\frac{s}{2})^{2n+1} } }-
\sum^{\infty}_{k=0}{ \frac{1} {(k+\frac{1+s}{2})^{2n+1} } }.
\biggr)
\end{split}
\end{equation}
This leads us to establish and define the new function Omega $\Omega(s)$ as

\begin{equation}
\begin{split}
\Omega(s) = 
\psi^{(2)}(\frac{s}{2}) + \psi^{(2)}(\frac{1-s}{2}) - 
\psi^{(2)}(1-\frac{s}{2}) - \psi^{(2)}(\frac{1+s}{2})
=\\
2^{2n}\frac{d^{2n+1}}{ds^{2n+1}} \biggr(
- \log(\frac{\sin(\frac{\pi (\frac{s+2}{2})}{2})}{\sin(\frac{\pi (1-\frac{s+1}{2})}{2})})
+  \log(\frac{\sin(\frac{\pi (\frac{s+1}{2})}{2})}{\sin(\frac{\pi (1-\frac{s}{2})}{2})})
-\log(\frac{\sin(\frac{\pi (\frac{s}{2})}{2})}{\sin(\frac{\pi (1-\frac{s-1}{2})}{2})})
+\log(\frac{\sin(\frac{\pi (\frac{s-1}{2})}{2})}{\sin(\frac{\pi (1-\frac{s-2}{2})}{2})})
\biggr)
= \\
-\Gamma(2n+1) \biggr(
\sum^{\infty}_{k=0}{ \frac{1} {(k+\frac{s}{2})^{2n+1} } }+
\sum^{\infty}_{k=0}{ \frac{1} {(k+\frac{1-s}{2})^{2n+1} } }-
\sum^{\infty}_{k=0}{ \frac{1} {(k+1-\frac{s}{2})^{2n+1} } }-
\sum^{\infty}_{k=0}{ \frac{1} {(k+\frac{1+s}{2})^{2n+1} } }
\biggr).
\end{split}
\end{equation}
such that 

\begin{equation}
\begin{split}
\Omega(s) = 
\psi^{(2n)}(\frac{s}{2}) + \psi^{(2n)}(\frac{1-s}{2}) - 
\psi^{(2n)}(1-\frac{s}{2}) - \psi^{(2n)}(\frac{1+s}{2})
=\\
\biggr( \psi^{(2n)}(\frac{s}{2}) - \psi^{(2n)}(\frac{1}{2}+\frac{s}{2}) \biggr)+ 
\biggr(\psi^{(2n)}(\frac{1-s}{2}) - 
 \psi^{(2n)}(\frac{1}{2}+\frac{1-s}{2}) \biggr)
=\\
(-2)^{2n+1}\Gamma(2n+1)\biggr(\sum^{\infty}_{k=0}{ \frac{(-1)^k} {(k+s)^{2n+1} } } +
\sum^{\infty}_{k=0}{ \frac{(-1)^k} {(k+1-s)^{2n+1} } } \biggr) =\\
2^{2n}\frac{d^{2n+1}}{ds^{2n+1}} \biggr(
- \log(\frac{\sin(\frac{\pi (\frac{s+2}{2})}{2})}{\cos(\frac{\pi (\frac{s+1}{2})}{2})})
+  \log(\frac{\sin(\frac{\pi (\frac{s+1}{2})}{2})}{\cos(\frac{\pi (\frac{s}{2})}{2})})
-\log(\frac{\sin(\frac{\pi (\frac{s}{2})}{2})}{\cos(\frac{\pi (\frac{s-1}{2})}{2})})
+\log(\frac{\sin(\frac{\pi (\frac{s-1}{2})}{2})}{\cos(\frac{\pi (\frac{s-2}{2})}{2})})
\biggr)
= \\
-\Gamma(2n+1) \biggr(
\sum^{\infty}_{k=0}{ \frac{1} {(k+\frac{s}{2})^{2n+1} } }+
\sum^{\infty}_{k=0}{ \frac{1} {(k+\frac{1-s}{2})^{2n+1} } }-
\sum^{\infty}_{k=0}{ \frac{1} {(k+1-\frac{s}{2})^{2n+1} } }-
\sum^{\infty}_{k=0}{ \frac{1} {(k+\frac{1+s}{2})^{2n+1} } }
\biggr)
\end{split}
\end{equation}

\end{proof}

The functional equations associated with this new function satisfy the following properties:
\begin{equation}
\begin{split} 
\Omega(s) = \Omega(1-s);\\
\Omega(s) = -\Omega(-s); \\
\Omega(s) = -\Omega(s \pm 1);\\
\Omega(s) = \Omega(s \pm 2).
\end{split}
\end{equation}
As a result the following fundamental properties emerge in general:
\begin{equation}
\begin{split} 
\Omega(s) = -\Omega(s \pm 2j+1) = \Omega(-s \pm 2j+1); \\
\Omega(s) = \Omega(s \pm 2j) = -\Omega(-s \pm 2j);
\end{split}
\end{equation}
where j is any integer number. This omega function has poles at integer points. For instance, at the point s=0 and s=1, the omega function, i.e. $\Omega(0)$  or $\Omega(1)$, requires the estimation of $\sum^{\infty}_{k=0}{ \frac{1} {k^{2n+1} } }= \frac{1}{0}+\zeta(2n+1)$:
\begin{equation}
\begin{split}
\Omega(0) = \Omega(1) =
-\Gamma(2n+1) \biggr(
\sum^{\infty}_{k=0}{ \frac{1} {k^{2n+1} } }+
\sum^{\infty}_{k=0}{ \frac{1} {(k+\frac{1}{2})^{2n+1} }}-
\sum^{\infty}_{k=0}{ \frac{1} {(k+1)^{2n+1} } }-
\sum^{\infty}_{k=0}{ \frac{1} {(k+\frac{1}{2})^{2n+1} } }
\biggr).
\end{split}
\end{equation}

\subsection{The omega function of golden ratio}
Let $\phi$ represent the golden ratio constant $ \frac{1+\sqrt{5}}{2}$. The implication of the omega functional equation is that $\Omega(\phi) = \Omega(1-\phi) = \Omega(1+\frac{1}{\phi}) = \Omega(-\frac{1}{\phi})$ because $ \phi = 1+(\phi-1) = 1+\frac{1}{\phi}$. This suggests that all the series sums below will produce the same result. 

$s \rightarrow \phi$:
\begin{equation}
\begin{split}
\frac{\Omega(\phi)}{-\Gamma(2n+1) }= 
\sum^{\infty}_{k=0}{ \frac{1} {(k+\frac{\phi}{2})^{2n+1} } }+
\sum^{\infty}_{k=0}{ \frac{1} {(k+\frac{1}{2}-\frac{\phi}{2})^{2n+1} } }-
\sum^{\infty}_{k=0}{ \frac{1} {(k+1-\frac{\phi}{2})^{2n+1} } }-
\sum^{\infty}_{k=0}{ \frac{1} {(k+\frac{1}{2}+\frac{\phi}{2})^{2n+1} } }.
\end{split}
\end{equation}

$s \rightarrow 1-\phi$:
\begin{equation}
\begin{split}
\frac{\Omega(1-\phi)}{-\Gamma(2n+1) }= 
\sum^{\infty}_{k=0}{ \frac{1} {(k+\frac{1}{2}-\frac{\phi}{2})^{2n+1} } }+
\sum^{\infty}_{k=0}{ \frac{1} {(k+\frac{\phi}{2})^{2n+1} } }-
\sum^{\infty}_{k=0}{ \frac{1} {(k+\frac{1}{2}+\frac{\phi}{2})^{2n+1} } }-
\sum^{\infty}_{k=0}{ \frac{1} {(k+1-\frac{\phi}{2})^{2n+1} } }.
\end{split}
\end{equation}

$s \rightarrow 1+\frac{1}{\phi}$:
\begin{equation}
\begin{split}
\textcolor{lightgray}{
\frac{\Omega(1+\frac{1}{\phi}) }{-\Gamma(2n+1) }= 
\sum^{\infty}_{k=0}{ \frac{1} {(k+\frac{1}{2}+\frac{1}{2\phi})^{2n+1} } }+
\sum^{\infty}_{k=0}{ \frac{1} {(k-\frac{1}{2\phi})^{2n+1} } }-
\sum^{\infty}_{k=0}{ \frac{1} {(k+\frac{1}{2}-\frac{1}{2\phi})^{2n+1} } }-
\sum^{\infty}_{k=0}{ \frac{1} {(k+1+\frac{1}{2\phi})^{2n+1} } }
} \\
\frac{\Omega(1+\frac{1}{\phi}) }{-\Gamma(2n+1) }
=
\sum^{\infty}_{k=0}{ \frac{1} {(k-\frac{1}{2\phi})^{2n+1} } }-
\sum^{\infty}_{k=0}{ \frac{1} {(k+\frac{1}{2}-\frac{1}{2\phi})^{2n+1} } }-
\sum^{\infty}_{k=0}{ \frac{1} {(k+1+\frac{1}{2\phi})^{2n+1} } }+
\sum^{\infty}_{k=0}{ \frac{1} {(k+\frac{1}{2}+\frac{1}{2\phi})^{2n+1} } }.
\end{split}
\end{equation}

$s \rightarrow -\frac{1}{\phi}$:
\begin{equation}
\begin{split}
\frac{\Omega(-\frac{1}{\phi})}{-\Gamma(2n+1) }= 
\sum^{\infty}_{k=0}{ \frac{1} {(k-\frac{1}{2\phi})^{2n+1} } }+
\sum^{\infty}_{k=0}{ \frac{1} {(k+\frac{1}{2}+\frac{1}{2\phi})^{2n+1} } }-
\sum^{\infty}_{k=0}{ \frac{1} {(k+1+\frac{1}{2\phi})^{2n+1} } }-
\sum^{\infty}_{k=0}{ \frac{1} {(k+\frac{1}{2}-\frac{1}{2\phi})^{2n+1} } }.
\end{split}
\end{equation}

\section{Discussion}
In this paper we derived a new formula for representing the actual value of $\psi^{2n}(s)$ from which we further derived another formula for expressing the value of $\zeta(2n+1)$ employing a method dependent on $\psi^{2n}(s)+\psi^{2n}(1-s)$ calculation. For instance, we uncovered the new formula for the first time

\[
\zeta(2n+1)= 
-{{( \psi^{(2n)}({1\over4})+\psi^{(2n)}({3\over4}))}\over{2^{2n+1}.(2^{2n+1}-1)}}.{1 \over{\Gamma(2n+1)}}
= 
-\frac{2
\frac{d^{(2n+1)}}{ds^{(2n+1)}}(\log(\frac{\zeta(1-s)}{\zeta(s)}))\biggr \vert_{s \rightarrow \frac{1}{4}}  + \frac{d^{(2n+1)}}{ds^{(2n+1)}}(\log(\tan(\frac{\pi s}{2}))) \biggr \vert_{s \rightarrow \frac{1}{4}} 
}
{
2^{2n+1}(2^{2n+1}-1)\Gamma(2n+1)
}
\]

We then presented new strategies for calculating the value of the logarithm of $\frac{\zeta(s)}{\zeta(1-s)}$ based on fractions of other zeta functions at smaller points of arguments. New identities relating the logarithm of Riemann zeta function to that of gamma function were also presented and all proved primarily using elementary functions only. We considered presenting new combinatorial results and perspectives on contemporay zeta functional analysis to establish deeper connections with several related zeta functions for the first time. We then used one of our main results about a recent discovery (also presented here) to uncover and present the relation 

\[
\begin{split}
\Omega(s) = 
\psi^{(2n)}(\frac{s}{2}) + \psi^{(2n)}(\frac{1-s}{2}) - 
\psi^{(2n)}(1-\frac{s}{2}) - \psi^{(2n)}(\frac{1+s}{2})
=\\
\biggr( \psi^{(2n)}(\frac{s}{2}) - \psi^{(2n)}(\frac{1}{2}+\frac{s}{2}) \biggr)+ 
\biggr(\psi^{(2n)}(\frac{1-s}{2}) - 
 \psi^{(2n)}(\frac{1}{2}+\frac{1-s}{2}) \biggr)
=\\
(-2)^{2n+1}\Gamma(2n+1)\biggr(\sum^{\infty}_{k=0}{ \frac{(-1)^k} {(k+s)^{2n+1} } } +
\sum^{\infty}_{k=0}{ \frac{(-1)^k} {(k+1-s)^{2n+1} } } \biggr) =\\
2^{2n}\frac{d^{2n+1}}{ds^{2n+1}} \biggr(
- \log(\frac{\sin(\frac{\pi (\frac{s+2}{2})}{2})}{\cos(\frac{\pi (\frac{s+1}{2})}{2})})
+  \log(\frac{\sin(\frac{\pi (\frac{s+1}{2})}{2})}{\cos(\frac{\pi (\frac{s}{2})}{2})})
-\log(\frac{\sin(\frac{\pi (\frac{s}{2})}{2})}{\cos(\frac{\pi (\frac{s-1}{2})}{2})})
+\log(\frac{\sin(\frac{\pi (\frac{s-1}{2})}{2})}{\cos(\frac{\pi (\frac{s-2}{2})}{2})})
\biggr)
= \\
-\Gamma(2n+1) \biggr(
\sum^{\infty}_{k=0}{ \frac{1} {(k+\frac{s}{2})^{2n+1} } }+
\sum^{\infty}_{k=0}{ \frac{1} {(k+\frac{1-s}{2})^{2n+1} } }-
\sum^{\infty}_{k=0}{ \frac{1} {(k+1-\frac{s}{2})^{2n+1} } }-
\sum^{\infty}_{k=0}{ \frac{1} {(k+\frac{1+s}{2})^{2n+1} } }
\biggr)
\end{split}
\]
which may provide new optimised and efficient strategies for analysing and evaluating sums of alternating series related to zeta, Dirichlet beta, and important nonelementary functions. For example, 

\begin{equation}
\begin{split}
-{ 2(2^{2n+1})2^{2n+1}\beta(2n+1)}=
\Omega(\frac{1}{4}) = 
2\psi^{(2n)}(\frac{1}{4}) - 2\psi^{(2n)}(\frac{3}{4})
=\\
(-2)^{2n+1}\Gamma(2n+1)\biggr(\sum^{\infty}_{k=0}{ \frac{(-1)^k} {(k+\frac{1}{2})^{2n+1} } } +
\sum^{\infty}_{k=0}{ \frac{(-1)^k} {(k+\frac{1}{2})^{2n+1} } } \biggr) =\\
2^{2n}\frac{d^{2n+1}}{ds^{2n+1}} \biggr(
- \log(\frac{\sin(\frac{\pi (\frac{s+2}{2})}{2})}{\cos(\frac{\pi (\frac{s+1}{2})}{2})})
+  \log(\frac{\sin(\frac{\pi (\frac{s+1}{2})}{2})}{\cos(\frac{\pi (\frac{s}{2})}{2})})
-\log(\frac{\sin(\frac{\pi (\frac{s}{2})}{2})}{\cos(\frac{\pi (\frac{s-1}{2})}{2})})
+\log(\frac{\sin(\frac{\pi (\frac{s-1}{2})}{2})}{\cos(\frac{\pi (\frac{s-2}{2})}{2})})
\biggr) \biggr \vert_{x \rightarrow \frac{1}{2}}
\end{split}
\end{equation}
It is a delight knowing that the same technique could be used to analyse slightly more complicated results such as 

\begin{equation}
\begin{split}
\Omega(i) = 
\psi^{(2n)}(\frac{i}{2}) + \psi^{(2n)}(\frac{1-i}{2}) - 
\psi^{(2n)}(1-\frac{i}{2}) - \psi^{(2n)}(\frac{1+i}{2}) = \\
(-2)^{2n+1}\Gamma(2n+1)\biggr(\sum^{\infty}_{k=0}{ \frac{(-1)^k} {(k+i)^{2n+1} } } +
\sum^{\infty}_{k=0}{ \frac{(-1)^k} {(k+1-i)^{2n+1} } } \biggr) =\\
(-2)^{2n+1}\Gamma(2n+1)\biggr(\sum^{\infty}_{k=0}{ \frac{(-1)^k} {(k+i)^{2n+1} } } -
\sum^{\infty}_{k=1}{ \frac{(-1)^k} {(k-i)^{2n+1} } } \biggr) =\\
(-2)^{2n+1}\Gamma(2n+1)\biggr(\frac{1}{i^{2n+1}}+\sum^{\infty}_{k=1}{ \frac{(-1)^k( (k-i)^{2n+1} -(k+i)^{2n+1})} {(k^2+1)^{2n+1} } } \biggr) =\\
2^{2n}\frac{d^{2n+1}}{ds^{2n+1}} \biggr(
- \log(\frac{\sin(\frac{\pi (\frac{s+2}{2})}{2})}{\cos(\frac{\pi (\frac{s+1}{2})}{2})})
+  \log(\frac{\sin(\frac{\pi (\frac{s+1}{2})}{2})}{\cos(\frac{\pi (\frac{s}{2})}{2})})
-\log(\frac{\sin(\frac{\pi (\frac{s}{2})}{2})}{\cos(\frac{\pi (\frac{s-1}{2})}{2})})
+\log(\frac{\sin(\frac{\pi (\frac{s-1}{2})}{2})}{\cos(\frac{\pi (\frac{s-2}{2})}{2})})
\biggr) \biggr \vert_{x \rightarrow i}
\end{split}
\end{equation}
which has instantaneously provided a stable mechanism of evaluating the series sum in terms of derivatives of the logarithm of elementary trogonometric functions; as a result, avoiding actual computations of $\psi^{(2n)}(\frac{i}{2}),\psi^{(2n)}(\frac{1-i}{2}), \psi^{(2n)}(1-\frac{i}{2})$ and $\psi^{(2n)}(\frac{1+i}{2})$ without requiring factorising $( (k-i)^{2n+1} -(k+i)^{2n+1})$. 
We then ask: how important is the result
\begin{equation}
\begin{split}
\frac{1}{i^{2n+1}}+
\sum^{\infty}_{k=1}{ \frac{(-1)^k( (k-i)^{2n+1} -(k+i)^{2n+1})} {(k^2+1)^{2n+1} } } \\ =
-\frac{
\frac{d^{2n+1}}{ds^{2n+1}} \biggr(
- \log(\frac{\sin(\frac{\pi (\frac{s+2}{2})}{2})}{\cos(\frac{\pi (\frac{s+1}{2})}{2})})
+  \log(\frac{\sin(\frac{\pi (\frac{s+1}{2})}{2})}{\cos(\frac{\pi (\frac{s}{2})}{2})})
-\log(\frac{\sin(\frac{\pi (\frac{s}{2})}{2})}{\cos(\frac{\pi (\frac{s-1}{2})}{2})})
+\log(\frac{\sin(\frac{\pi (\frac{s-1}{2})}{2})}{\cos(\frac{\pi (\frac{s-2}{2})}{2})})
\biggr) \biggr \vert_{x \rightarrow i}
}
{2\Gamma(2n+1)} ~{\LARGE ?}
\end{split}
\end{equation}
We ponder on such interesting questions and sense the need to review the result

\begin{equation}
\begin{split}
\frac{1}{(-6i)i^3}+
\sum^{\infty}_{k=1}{ \frac{(-1)^k( k^2-\frac{1}{3})} {(k^2+1)^{3} } }\\=
-\frac{
\frac{d^{3}}{ds^{3}} \biggr(
- \log(\frac{\sin(\frac{\pi (\frac{s+2}{2})}{2})}{\cos(\frac{\pi (\frac{s+1}{2})}{2})})
+  \log(\frac{\sin(\frac{\pi (\frac{s+1}{2})}{2})}{\cos(\frac{\pi (\frac{s}{2})}{2})})
-\log(\frac{\sin(\frac{\pi (\frac{s}{2})}{2})}{\cos(\frac{\pi (\frac{s-1}{2})}{2})})
+\log(\frac{\sin(\frac{\pi (\frac{s-1}{2})}{2})}{\cos(\frac{\pi (\frac{s-2}{2})}{2})})
\biggr) \biggr \vert_{x \rightarrow i}
}
{2(-6i)\Gamma(2n+1)}.
\end{split}
\end{equation}
which was derived after just a few additional steps.

%\bibliography{Biblio}

{~\\
Dr Michael A. Idowu\\
University of Abertay\\
Dundee \\
DD1 1HG\\
Scotland\\
United Kingdom\\
m.idowu@abertay.ac.uk; \\
michade@hotmail.com}

\end{document}